\documentclass[a4paper,12pt,oneside]{article}

\usepackage{amsmath,amsthm,amsfonts,amssymb,oldgerm}
\usepackage[frenchb,english]{babel}
\usepackage{indentfirst}
\numberwithin{equation}{section}
\usepackage[citecolor=blue,colorlinks=true]{hyperref}
\usepackage{color}
\allowdisplaybreaks[1]

\usepackage[top=3cm, left=2cm, bottom=3cm, right=2cm]{geometry}

\usepackage[utf8x]{inputenc}
\usepackage[all]{xy}
\usepackage{graphicx}
\usepackage{pict2e}

\newcommand{\disp}{\displaystyle}
\newcommand{\p}{\partial}

\newcommand{\E}{\mathds{E}}
\renewcommand{\H}{\mathds{H}}
\newcommand{\N}{\mathds{N}}
\renewcommand{\P}{\mathds{P}}

\newcommand{\R}{\mathds{R}}

\newcommand{\Z}{\mathds{Z}}

\renewcommand{\1}{\mathds{1}}

\newcommand{\bH}{\bold{H}}

\newcommand\cA{{\mathcal A}}

\newcommand\cD{{\mathcal D}}

\newcommand\cF{{\mathcal F}}

\newcommand\cL{{\mathcal L}}

\newcommand\cP{{\mathcal P}}

\newcommand\cR{{\mathcal R}}

\newcommand\fB{{\mathfrak B}}
\newcommand\fD{{\mathfrak D}}

\usepackage{dsfont}				
\usepackage{cleveref}

\newtheorem{theorem}{Theorem}[section]
\crefname{theorem}{Theorem}{Theorems}

\crefname{proof}{Proof}{Proofs}

\newtheorem{proposition}{Proposition}[section]
\crefname{proposition}{Proposition}{Propositions}

\newtheorem{lemma}{Lemma}[section]
\crefname{lemma}{Lemma}{Lemmas}

\crefname{equation}{equation}{equations}

\crefname{section}{Section}{Sections}

\crefname{remark}{Remark}{Remarks}
\newtheorem{definition}{Definition}[section]
\crefname{definition}{Definition}{section}

\newtheorem{assumption}{Assumption}[section]
\crefname{assumption}{Assumption}{Assumptions}

\crefname{hypothesis}{Hypothesis}{Hypothesis}

\crefname{example}{Example}{Example}

\crefname{note}{Note}{Notes}

\crefname{corollary}{Corollary}{Corollary}

\renewcommand{\abstractname}

\begin{document}
	\title{Spatial modeling of cholera epidemic : \\ A Law of Large Numbers}
	\author{Mac Jugal Nguepedja Nankep , Boris Kouegou Kamen}
	\maketitle

\begin{abstract}
In this paper we propose a Stochastic model for studying a spatial cholera epidemic spreading where communities (Humans and Bacteria)  are spatially distributed on a one-dimensional lattice and the bacteria are transported along a network links that are thought as the hydrological connection in the studied area. We prove a Law of Large numbers which  suggest that  in large communities (both Human and Bacteria) the stochastic model  behave as a deterministic spatial  model  proposed and studied in the literature by \cite{BerCasGatRodRin2010}  and therefore it is    quite unavoidable  to ask how large fluctuations effect can occur between these two version. That question will be treated in a forthcoming work as large deviations estimates.   We also discuss at the end of the work different possible scaling that could be treated using similar mathematical tools and which lead to different limits.  
\end{abstract}

\section{Introduction}

Until now Deterministic models have always been used when studying infectious disease outbreak dynamics and formulating outbreak response options. As it is well known all epidemic models are inherently stochastic at the level of individuals and then it is  more realistic to have stochastic models.
We are interested in modelling the spreading of Cholera in an endemic area where it is usually known there is little hydraulic hygiene. Cholera is an  acute intestinal infection  caused by a bacterium called \textit{Vibro choler\ae} which is most commonly transmitted orally by ingestion of water or contaminated food, or by contact with any liquid of an infected host (saliva, sweat, etc...). 
\textit{Vibro choler\ae} lives and spreads in the water where she's endowed with an incredible capacity for survival. Rivers, streams and groundwater or any source of water contaminated by Human defects are its  preferred reservoirs. 
Several deterministic works  showed the importance of the spatial distribution of humans and reservoirs in the spread of this disease in endemic areas. 


A deterministic model is proposed by \cite{BerCasGatRodRin2010}, where spatial dynamic is effectively considered for bacteria only, through water reservoirs. As their spatial transport is asymmetric, bacteria population evolves according to a reaction-advection-diffusion partial differential equation. The model is of type SIRSB, where susceptible (S) can be infected directly by contact with bacteria (B) or infectious (I). Infectious can recover and are called recovered (R). These latters are immune within a period. In \cite{BerCasGatRodRin2010}, the authors focus on the epidemic period of the disease, which is shorter than the period of immunity, so that they need not consider the compartment (R) of recovered people. Furthermore, bacteria transportation are oriented and their probabilities and rates depend on the directions.

In this chapter, the stochastic counterpart of the model of \cite{BerCasGatRodRin2010} is developped. We study its large population and long time asymptotic behaviours. With respect to the long time behaviour, we explicitly consider the compartment (R) of recovered people, since they can loose their immunity with time, and become susceptible again. The model has two population time scales, the one of  humans denoted $H$, 
and the one of bacteria denoted $K$.
The main results are stated under the condition that $H$ scales with $K$. 
However, other possibilities are discussed and other deterministic limits are identified as well as hybrid limits.

The rest of the chapter is organised as follows. Section 2 is devoted to modeling. After recalling the model of \cite{BerCasGatRodRin2010}, we present our model, which is its stochastic counterpart. In section 3, a law of large numbers (LLN) is established. Relying on chapter 2 or \cite{Blount1992}, we prove that the stochastic model well renormalized converges, in large populations limit, to its deterministic corresponding, in the supremum norm. In section 4, a corresponding large deviation principle (LDP) is investigated. We first present a rate function candidate. Then, we proceed to the upper 
bound estimates. Technical computations are posponed to the Appendix section at the end.

\hspace{1cm}\\

\noindent \underline{\textbf{\textit{Some general notations.}}} 
Let $(Z,\Vert\cdot\Vert_Z)$ and $(\tilde{Z},\Vert\cdot\Vert_{\tilde{Z}})$ be Banach spaces. The product space $Z\times\tilde{Z}$ is equipped with the norm $\Vert\cdot\Vert_Z+\Vert\cdot\Vert_{\tilde{Z}}$. We introduce:
\begin{description}
\item[$\bullet$] $\mathcal{L}(Z,\tilde{Z})$: the space of continuous linear maps from $Z$ to $\tilde{Z}$. If $Z=\tilde{Z}$, one simply writes $\mathcal{L}(Z)$.
The operator norm is denoted $\Vert\cdot\Vert_{Z\rightarrow\tilde Z}$ and when there is no risk of confusion, we denote it $\Vert\cdot\Vert$.


\item[$\bullet$] $\fB(Z)$ (resp. $\fB_b(Z)$): the space of Borel-measurable (resp. bounded Borel-measurable) real valued functions on $Z$. The space $\fB_b(Z)$ is endowed with the supremum norm \[ \Vert f\Vert_{\fB_b(Z)}=\sup_{x\in Z}|f(x)|=\Vert f\Vert_\infty. \]

\item[$\bullet$] $\disp C_b^k(Z)$, $k\in \N$: the space of real valued functions of class $C^k$, i.e. $k$-continuously Fr\'echet differentiable,  on $Z$ which are bounded and have uniformly bounded succesive differentials. It is equipped with the norm 
\[ \Vert f\Vert_{C_b^k(Z)}=\sum_{i=0}^k\big\Vert D^{i}f\big\Vert_\infty=:\sum_{i=0}^k\Vert f\Vert_{i,\infty}, \] 
where $D^{i}f$ is the $i$-th differential of $f\in C_b^k(Z)$, $\disp C_b^0(Z)=C_b(Z)$ is the set of bounded continuous real valued functions on $Z$, and we are using the notation $\Vert f\Vert_{i,\infty}:=\Vert D^{i}f \Vert_\infty$. 


\item[$\bullet$] $C^{l,k}(Z\times\tilde{Z})$, $l,k\in \N$: the set of real valued functions $\varphi$ of class $C^l$ w.r.t\footnote{with respect to}. the first variable and of class $C^k$ w.r.t. the second. In particular, $C^{0,0}(Z\times\tilde{Z)})=C(Z\times\tilde{Z)})$. \par 

For $(z,\tilde{z})\in Z\times\tilde{Z}$, we donote by $D^{l,k}\varphi(z,\tilde{z})$ the (Fr\'echet) differential of $\varphi$, of order $l$ w.r.t. $z$ and of order $k$ w.r.t. $\tilde{z}$, computed at $(z,\tilde{z})$.\par  

Also, a subscript $b$ can be added - to obtain $C_b^{l,k}(Z\times \tilde{Z})$ - in order to specify that the functions and their succesive differentials are uniformly bounded. 
 

\item[$\bullet$] $\disp C(J)$ (resp. $C^k(J)$): the set of periodic continuous (resp. $C^k$) real valued functions defined on $J=[0,1]$.

\item[$\bullet$] $\bold C(J):=C(J)\times C(J)\times C(J) \times C(J)$. 

\item[$\bullet$] $\disp C_p(J)$: the set of piecewise continuous real valued functions defined on $J=[0,1]$. It is equipped with the supremum norm.

\item[$\bullet$] $\bold C_p(J):=C_p(J)\times C_p(J)\times C_p(J) \times C_p(J)$.

\item[$\bullet$] $L^2(J)$: the set of square integrable real valued functions defined and 1-periodic on $J=[0,1]$. It is endowed with its usual inner product \[ \langle f,g\rangle_2=\int_Jf(x)g(x)dx\hspace{0.5cm} \text{and the induced norm}\hspace{0.5cm} \Vert f\Vert_2=\sqrt{<f,f>_2}. \] 

\item[$\bullet$] $\bold L^2(J):=L^2(J)\times L^2(J)\times L^2(J)\times L^2(J)$ is a Hilbert space with natural inner product 
\[ \langle f,g\rangle_{\bold 2}:=\sum_{j=1}^4\langle f_j,g_j\rangle_2 \hspace{0.5cm}\text{and the induced norm}\hspace{0.5cm}\Vert f\Vert_{\bold 2}:=\sum_{j=1}^4\Vert f_j\Vert_2, \] 
defined for $f=\big(f_1,\cdots,f_4\big), g=\big(g_1,\cdots,g_4\big)\in \bold L^2(J)$. 



\item[$\bullet$] $\disp D\big(\R_+,Z\big)$: the set of \textit{c\`adl\`ag} processes defined on $\R_+$ and taking values in $Z$. It is endowed with the Skorohod topologie. 
 
\end{description}

\section{Modeling and convergence tools}

We first present existing deterministic models as given in \cite{BerCasGatRodRin2010}, where the authors study the disease spread in a population of Humans interacting with a population of bacteria in reservoirs, in some endemic area $J\subset \R^d$. We consider a one dimensional spatial domain - $d=1$ - and take $J=[0,1]$, the unit interval. We also consider periodic boundary conditions. That is we are interested in functions that are $1$-periodic w.r.t. the space variable. 

Given any spatio-temporal coordinate $(t,x)\in\R_+\times J$, human living population is subdivided into classes, according to their disease status: \vspace{0.1cm}\par 
$\bullet$ $S(t,x)$ : is the number of humans who are susceptible to catch the disease,\par 
$\bullet$ $I(t,x)$ : is the number of humans who are infected and therefore are infectious too,\par 
$\bullet$ $R(t,x)$ : is the number of infected humans who received treatment and are still alive. \vspace{0.1cm}\par 

\noindent Human total population is $\bH(t,x)=S(t,x)+I(t,x)+R(t,x)$. Its average initial value  \[ H:=\int_J\bH(0,x)dx \] will turn out to be of particular importance. 
In addition, we denote by $B(t,x)$ the population of \textit{ vibrio choler\ae} for every time-space coordinate $(t,x)\in\R_+\times J$. Bacteria live in water reservoirs. Commonly, $S$, $I$, $R$ and $B$ are refered to, as the compartments of the model. We often use the notation $S(t)=S(t,\cdot)$, and similar notation for the other compartments. The functions $S(t)$, $I(t)$, $R(t)$ and $B(t)$ are considered to be periodic, of period $1$.

In \cite{BerCasGatRodRin2010}, the class $R$ of recovered is not considered, since the authors focus on the population behavior over the epidemic period, and this latter is less important than the time before the loss of immunity. Contrariwise, we have to consider that class, as we are interested in large deviations study, which involves a long time behavior.


\subsection{Deterministic model}\label{DeterministicModel}

\noindent \textbf{\textit{Homogeneous model.}} 
Only a global description of the system is needed, and the quantities of interest are concentrations. Human and bacteria are assumed to be spatially homogeneous, meaning that for all $(t,x)\in \R_+\times J$, one has $S(t,x)=S(t)$. A similar relation holds for $I$, $R$, and $B$. In this context, the human average initial population is simply the human initial total population: \[ H=\bH(0). \] 
The figure below describes the transition mechanisms that the different compartments undergo.
 \begin{center}    
 	\setlength{\unitlength}{0.45cm}
 	\linethickness{0.2mm}
 	\dashbox(28,17)
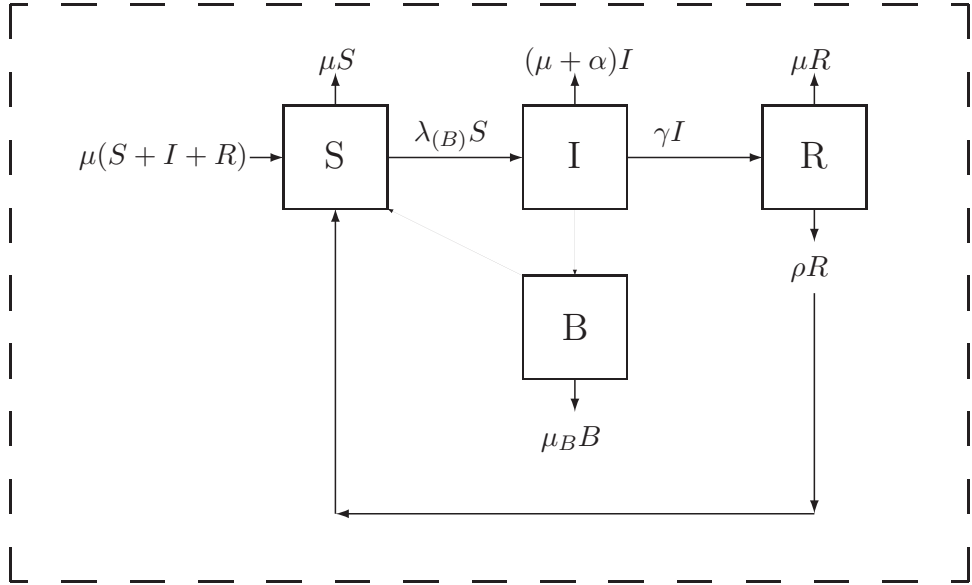
{
 		\begin{picture}(20,17)(1,-10)
 		\put(4,2.5){\vector(1,0){1}}
 		\put(5,1){\framebox(3,3){\large{$\text{S}_{ }$}}} 
 		\put(8,2.5){\vector(1,0){4}}
 		\put(12,1){\framebox(3,3){\large{$\text{I}_{ }$}}}
 		\put(15,2.5){\vector(1,0){4}}
 		\put(19,1){\framebox(3,3){\large{$\text{R}_{ }$}}}
 		\put(8.8,3){\small $\lambda_{(B)} S$} 
 		\put(15.8,3){\small $\gamma I$} 
 		\put(12,-4){\framebox(3,3){\large{$\text{B}$}}}
 		
 		\put(6.5,4){\vector(0,1){1}}
 		\put(13.5,4){\vector(0,1){1}}
 		\put(20.5,4){\vector(0,1){1}}  
 		\put(20.5,1){\vector(0,-1){1}}				 
 		\put(19.8,-1){\small $\rho R$}
 		\put(20.5,-1.5){\vector(0,-1){6.5}}        
 		\put(20.5,-8){\vector(-1,0){14}}
 		\put(6.5,-8){\vector(0,1){9}}
 		%
 		\put(13.5,-4){\vector(0,-1){1}}
 		
 		\linethickness{0.01mm}
 		\put(12,-1){\vector(-4,2){4}}           
		\put(13.5,1){\vector(0,-1){2}} 
 		\put(-1,2.3){\small $\mu(S+I+R)$}  
 		\put(6,5.1){\small $\mu S$}          
 		\put(12,5.1){\small $(\mu+\alpha)I$}			
 		\put(19.8,5.1){\small $\mu R$} 
 		
 		\put(12.5,-6){\small $\mu_B B$} 
 		\end{picture} }
 \end{center}

\noindent Here, \vspace{0.1cm} \par 

$\bullet$ $\mu$ is the natural death and birth rate per human. \par 
$\bullet$ $\alpha$ is the cholera-mortality rate per infected.\par 
$\bullet$ $\gamma$ is the rate  at which infected recover health.\par 
$\bullet$ $\rho$ is the rate at which a recovered loses his immunity.\par 
$\bullet$ $p/W$ is the rate at which an infected contributes to the concentration of vibrios. Bacteria produced by an infected person reach and contaminate a water reservoir of volume $W$, at rate $p$.\par 
$\bullet$ $\lambda_{(B)}(t) = \beta \frac{B(t)}{K + B(t)}$ is the rate at which susceptible people become infected. Therein, $\beta$ is the rate of contacts with contaminated water per susceptible, $K$ is the capacity of bacteria concentration in the area, and $\frac{B(t)}{K + B(t)}$ is the logistic dose-response curve. Such a curve links the probability of becoming infected to the concentration of vibrios B in water (see \cite{Codeco2001}).\par 
$\bullet$ $\mu_B$ is the death rate  per \textit{vibrio choler\ae}. This parameter actually takes into account both the reproduction and the death of free-living vibrios. However, these latter reproduce in water at a smaller rate than that of their mortality. So basically, they "only" die.\vspace{0.2cm}\par

In such context, the spread of the disease is described by the system of ordinary differential equations (ODEs) of concentrations:
\begin{equation}\label{FirstOde}
 \left\{ \begin{array}{lcl}
			\displaystyle \frac{dS(t)}{dt} &=& 	\displaystyle     \mu I(t) + (\mu+\rho)R(t) 	- \lambda_{(B)}(t)S(t) \vspace{0.2cm}\\
			\displaystyle \frac{dI(t)}{dt} &=& \displaystyle \lambda_{(B)}(t)S(t)  - (\gamma + \alpha + \mu)I(t) \vspace{0.2cm}\\
			\displaystyle \frac{dR(t)}{dt} &=& \displaystyle \gamma I(t) - (\mu+\rho)R(t) \vspace{0.2cm}\\
			\displaystyle \frac{dB(t)}{dt} &=& \displaystyle -\mu_B B(t) + \frac{p}{KW}I(t).
		  \end{array}
 \right.
\end{equation}

\noindent \textbf{\textit{Renormalization.}} Proportions are often described in practice, in the place of the number of individuals. One switches from the latter to the former by setting 
\[ S^*(t)=\frac{S(t)}{H},\hspace{1cm} I^*(t)=\frac{I(t)}{H}, \hspace{1cm} R^*(t)=\frac{R(t)}{H}\hspace{1cm}\text{and}\hspace{1cm}B^*(t)=\frac{B(t)}{K}. \] 
As a result, the rates of events for the rescaled variables are of order $H$ (resp. $K$) for humans (resp. bacteria). The rescaled variables also depend on the parameters $H$ and $K$, but we do not mention these latter throughout this work. The renormalized version of \eqref{FirstOde} is given by
\begin{equation}\label{FirstOdeRenormalized}
 \left\{ \begin{array}{lcl}
			\displaystyle \frac{dS^*(t)}{dt} &=& 	\displaystyle     \mu I^*(t)) + (\mu+\rho)R^*(t) 	- \lambda_{(B^*)}(t)S^*(t) \vspace{0.2cm}\\
			\displaystyle \frac{dI^*(t)}{dt} &=& \displaystyle \lambda_{(B^*)}(t)S^*(t)  - (\gamma + \alpha + \mu)I^*(t) \vspace{0.2cm}\\
			\displaystyle \frac{dR^*(t)}{dt} &=& \displaystyle \gamma I^*(t) - (\rho + \mu)R^*(t) \vspace{0.2cm}\\
			\displaystyle \frac{dB^*(t)}{dt} &=& \displaystyle -\mu_B B^*(t) + \frac{H p}{KW}I^*(t),
		  \end{array}
 \right.
\end{equation}
where $\lambda_{(B^*)}(t)=\beta\frac{B^*(t)}{1+B^*(t)}$.\\

\noindent\textbf{\textit{Spatial model.}} 
In this context, a local description is given, and compartments are now functions of the space variable. Periodic boundary conditions are considered. We discretize $J$ w.r.t. the well known one dimensional lattice regular subdivision in $N$ parts. There are $N$ nodes, indexed by $i$, and the subintervals $J_i=\big((i-1)/N,i/N\big]$, $1\leq i\leq N$ are called sites. We view each node as a site which has been concentrated at a point in such a way that the distance between two neighboring nodes is constant and equal to the length of each site. We equivalently consider the two notions and say node as well as site. \par 

In accordance with the precedent notation, $S_i$, $I_i$,  $R_i$, and $B_i$ are the number of susceptible, infectious , recovered, and bacteria on node $i$, respectively. Note that these quantities actually depend on the parameter $N$ of the subdivition. Also, on a node $i$, the human total population is $\bH_i$, and the water volume is $W_i$. Periodicity at the boundary allows us to consider that $S_{i+N}=S_i$ for all $i\in \Z$. The same holds for $I_i$, $R_i$, and $B_i$.

As previously, proportions are captured by rescaling. The average initial human population on each site is 
\[ H=\frac{1}{N}\sum_{i=1}^N\bH_i(0)=\frac{\bH(0)}{N}, \] 
and we set $\disp S_i^{*}(t)=S_i(t)/H$, $\disp I_i^{*}(t)=I_i(t)/H$, $\disp R_i^{*}(t)=R_i(t)/H$, and, $\disp B_i^{*}(t)=B_i(t)/K$. Henceforth, only rescaled variables are considered. Thus, we forget about the superscript * in our notation and write $S_i$, $I_i$, ... , instead of $S^{*}$, $I^{*}$, ... We make the following assumption.

\begin{assumption}\label{LatticeOrientation}  \hspace{1cm} \vspace{0.1cm}\par 
(i) The lattice is oriended and we fix an orientation, in order to perform precise computations: the direction of the edges follows the increasing numerical order of nodes index \[ \cdots (i-1) \longrightarrow i \longrightarrow (i+1)\cdots \]
\indent (ii) Vibrios can move with a certain probability from a node to a connected node, through an inward or an outward edge, at a certain rate $\ell$.
\end{assumption}

\Cref{LatticeOrientation} is not restrictive from regarding to the space dimension. In fact, as in \cite{BerCasGatRodRin2010}, the following tools are usable in a higher spatial dimensional context. We define: \vspace{0.1cm}\par 
$\bullet$ $d_{in}(i)$ the number of inward edges of node $i$,\par 
$\bullet$ $d_{out}(i)$ the number of outward edges of node $i$,\par 
$\bullet$ $\mathcal{P}_{in}$ the transmission probability by an inward edge,\par 
$\bullet$ $\mathcal{P}_{out}$ the transmission probability by an outward edge. \vspace{0.1cm}\par 

\noindent Clearly, $d_{in}(i)=1=d_{out}(i)$ for all $i=1,\cdots,N$ in a one dimensional framework.
 
The probability that a  propagule transits from a node $i$ to another one $j$ has the form: 
  \begin{equation}\label{TransitionProbabilities}
  \mathcal{P}_{ij} = \left\{ 
  \begin{array}{ll}
	\displaystyle \frac{\mathcal{P}_{out}}{\mathcal{P}_{out}d_{out}(i) + \mathcal{P}_{in} d_{in}(i)} & \text{ if } i\to j \vspace{0.2cm}\\
	\displaystyle  \frac{\mathcal{P}_{in}}{\mathcal{P}_{out}d_{out}(i) + \mathcal{P}_{in}d_{in}(i) } & \text{ if } i \leftarrow j \vspace{0.2cm}\\
	0 &\text{ otherwise. }
\end{array} \right.
 \end{equation}
Of course, we have $\sum_{j=1}^{N}\mathcal{P}_{ij} =1$, since $\cP_{in}+\cP_{out}=1$. \par 
 
Infectious propagules are removed at every node -- at rate $\ell$ -- and transported through the network following the transition probabilities \eqref{TransitionProbabilities}. Hence, replacing probabilities with frequencies, the corresponding -- in this spatial context -- to the rescaled system \eqref{FirstOdeRenormalized} reads
\begin{equation}
 \left\{  \begin{array}{lcl}
			   \displaystyle \frac{dS_i(t)}{dt} &=&\displaystyle  \mu I_i(t) + (\mu+\rho)R_i(t) - \beta\frac{B_i(t)}{1 + B_i(t)}S_i(t) \vspace{0.2cm} \\
			   \displaystyle \frac{dI_i(t)}{dt} &=& \displaystyle \beta \frac{B_i(t)}{1 + B_i(t)}S_i(t)  - (\gamma + \alpha + \mu)I_i(t) \vspace{0.2cm}\\
			   \displaystyle \frac{dR_i(t)}{dt} &=& \displaystyle \gamma I_i(t) - (\mu+\rho)R_i(t) \vspace{0.2cm}\\
			   \displaystyle \frac{dB_i(t)}{dt} &=& \displaystyle -\mu_B B_i(t) + \frac{Hp}{KW_i}I(t) - \ell B_i(t)  + \sum_{j=1}^N\ell\mathcal{P}_{ji}B_j(t)\frac{W_j}{W_i}.
			 \end{array}
 \right.
\end{equation}

In order to derive a continuous-space model, we introduce $b=P_{out}-P_{in}=2P_{out}-1$, the bias of the transport to follow the edge direction (the reader is refered to \cite{BerCasGatRodRin2010} and references therein, for more details about the parameter $b$). We take the limit $N\rightarrow\infty$ under the conditions:

\begin{assumption}\label{ConstantReservoirOnSitesAndConvergenceConditions}\hspace{1cm}\par 
(i) Water reservoir volume is constant on each site: $W_i=W$ for all $ i=1,\cdots,N$.\par 
(ii) The transport rate $\ell$ scales with $N^2$ and we put $\fD=\ell/2N^2$.\par 
(iii) The product $b\cdot\ell$ is of order $N$ and we put $\nu=b\cdot\ell/N$.
\end{assumption}

Then, as the site length goes to zero, we obtain the limit system of partial differential equations (PDEs) 
\begin{equation}\label{RenormalizedDerterministicPDE}
 \left\{  \begin{array}{lcl}
			   \displaystyle \frac{\p S(t,x)}{\p t} &=&\displaystyle  \mu I(t,x) + (\mu+\rho) R(t,x) - \beta\frac{B(t,x)}{1 + B(t,x)}S(t,x) \vspace{0.2cm} \\
			   \displaystyle \frac{\p I(t,x)}{\p t} &=& \displaystyle \beta\frac{B(t,x)}{1 + B(t,x)}S(t,x)  - (\gamma + \alpha + \mu)I(t,x) \vspace{0.2cm}\\
			   \displaystyle \frac{\p R(t,x)}{\p t} &=& \displaystyle \gamma I(t,x) - (\rho + \mu)R(t,x) \vspace{0.2cm}\\
			   \displaystyle \frac{\p B(t,x)}{\p t} &=& \displaystyle -\mu_B B(t,x) + \frac{H p}{KW}I(t,x) - \nu \frac{\p B(t,x)}{\p x} + \fD  \frac{\p^2 B(t,x)}{\p x^2}, \end{array}
\right.
\end{equation}
where $\fD>0$ and $\nu>0$ are respectively the diffusion coefficient and the advection velocity of the bacteria.\\

\noindent\textbf{\textit{Well-posedness, generator and debit function.}} A compact form of \eqref{RenormalizedDerterministicPDE} reads
\begin{equation}\label{CompactRenormalizedDeterministicPDE}
\disp \frac{dv(t)}{dt}= \tilde{A} v(t)+F\big(v(t)\big),
\end{equation}

where $v=(S,I,R,B)$, $F= (F_S,F_I, F_R, F_B)$ is the vector field in $\R^4$ given for $y\in(y_1,y_2,y_3,y_4)\in\R^4$ by
\[ 
F(y)=
\begin{pmatrix}
\disp F_S(y) \vspace{0.1cm}\\ \disp F_I(y) \vspace{0.1cm}\\ \disp F_R(y) \vspace{0.1cm}\\ \disp F_B(y)
\end{pmatrix} 
=
\begin{pmatrix}
\disp \mu y_2 + (\mu+\rho)y_3 -\beta\frac{y_4}{1+y_4}y_1\vspace{0.1cm}\\ 
\disp \beta\frac{y_4}{1+y_4}y_1 - (\gamma + \alpha + \mu)y_2 \vspace{0.1cm}\\ 
\disp \gamma y_2 - (\mu+\rho)y_3 \vspace{0.1cm}\\ 
\disp -\mu_B y_4 + \frac{H p}{KW}y_2
\end{pmatrix},
\]
and, $\tilde{A}$ is the diagonal matrix operator of size $4$ on $\bold L^2(J)$, given by 
\[ \tilde{A}=\text{diag}(0,0,0,A)\hspace{1cm}\text{with}\hspace{1cm} A=-\nu\nabla+\fD\Delta. \] 
In our notation, $0$ is the identically zero operator, $\nabla$ is the gradient, and $\Delta$ is the Laplace.

\medskip
We say a $\mathds R^4-$vector $y$ is positive, and we write $y\geq0$, when $y_i\geq0$ for $i=1,\cdots,4$. For consistency, only positive initial conditions are considered. Let $\R_+^4:=\big\{y\in\R^4: y\geq0\big\}$. Then it is not difficult to see that the restriction of $F$ to $\R_+^4$ is of class $C^\infty$. Furthermore, $F$ satisfies
\begin{equation}\label{ReactionDebitBoundProperties}
\begin{array}{ll}
(i) & \disp F_S(y)\geq 0 \hspace{0.2cm}\text{if}\hspace{0.2cm} y_1= 0,\hspace{0.2cm}\forall y=(y_1,\cdots,y_4)\in\R_+^4.\vspace{0.1cm}\\
(ii) & \disp \text{At (i), } (S,y_1) \text{ can be replaced by }(I,y_2), \hspace{0.1cm} (R,y_3) \text{ or } (B,y_4).\vspace{0.2cm}\\
(iii) & \disp \exists M>0, \hspace{0.2cm}\text{such that}\hspace{0.2cm} |F(y)|\leq M|y| \hspace{0.2cm}\forall y\geq0,
\end{array}
\end{equation}
where $|\cdot|$ is the norm of $\R^4$, and the constant $M$ depends on $\mu$, $\rho$, $\beta$, $\alpha$, $\gamma$, $\rho$, $H/K$ and $p/W$. Relations (i) and (ii) are trivial. Concercing (iii), for $y=(y_1,\cdots,y_4)\geq0$, $y_4\geq0$ yields $\frac{y_4}{1+y_4}\leq 1$. Then, one easily sees that there exists constants $M_S=M_S(\mu,\rho,\beta)$, $M_I=M_I(\beta,\gamma,\alpha,\mu)$, $M_R=M_R(\gamma,\rho,\mu)$ and $M_B=M_B\big(\mu_B,\frac{Hp}{KW}\big)$, such that $|F_S(y)|\leq M_S|y|$, $|F_I(y)|\leq M_I|y|$, $|F_R(y)|\leq M_R|y|$ and $|F_B(y)|\leq M_B|y|$, for all $y\geq0$.

\medskip
In addition, it is well known (see e.g. \cite{Pazy1983} Chapter 7, Theorem 3.7) that the differential operator $A$ on $C(J)$ is the generator of an analytic semigroup $\big\{ T(t)=e^{At},t\geq0\big\}$ which is bounded, uniformly in $t$.
%
%
Therefore, it is not difficult to see that $\tilde{A}$ is also the generator of an analytic semigroup 
on $\bold C(J)$, and we denote it $\tilde{T}(t)=e^{\tilde{A}t}$. This latter is also bounded uniformly in $t$. Let $c_1>0$ be such that $\Vert T(t)\Vert \leq c_1$ for all $t\geq 0$, where $\Vert\cdot\Vert$ stands for the operator norm. The following result of existence and uniqueness is derived then.

\begin{proposition}\label{Well-posednessOfCompactRenormalizedDeterministicPDE}
Consider \eqref{CompactRenormalizedDeterministicPDE} on $\bold C(J)$, with an initial condition $v(0)=v_0\in \big[C(J)\big]^3\times C^3(J)$ such that $v_0\geq0$ and $\Vert v_0\Vert_{\bold C(J)}<c_0$ for some $c_0\gg1$. We also consider periodic boundary conditions: $v_0(t,0)=v_0(t,1)$ for all $t\geq0$. 
Then, the described Cauchy problem has a unique global mild solution $v\in C\big(\R_+; \big[C(J)\big]^3\times C^3(J)\big)$ satisfying
\begin{equation}\label{MildRepresentationOfDeterministicModel}
v(t)=\tilde{T}(t)v_0+\int_0^t\tilde{T}(t-s)F\big(v(s)\big)ds
\end{equation}
and $v(t)\geq0$ for all $t\geq0$, and, for all $T>0$, 
\begin{equation}\label{BoundOfDeterministicModel}
\sup_{t\leq T}\Vert v(t)\Vert_{\bold C(J)}<c_T=c_1c_0\text{e}^{c_1MT}.
\end{equation}
\end{proposition}


\begin{proof}\normalfont
Positivity follows from \eqref{ReactionDebitBoundProperties} (i)-(ii), $v_0\geq0$ and the continuity of any solution. We restrict $F$ to $\R_+^4$. Thus, $F$ is locally Lipschitz and 
from \eqref{ReactionDebitBoundProperties} (iii), there exists a unique local mild solution to \eqref{CompactRenormalizedDeterministicPDE}, which lies in $C\big([0,T],\bold C(J)\big)$ for all $T\in (0,T(c_0)]$ for some $T(c_0)>0$. This follows using a Banach fixed point argument. 
Then, we obtain the bound \eqref{BoundOfDeterministicModel} through a Gronwall-Bellman lemma, thanks to \eqref{ReactionDebitBoundProperties} (iii). Therefore, the solution is global. Finally, the solution actually lies in $C\big(\R_+;\big[C(J)\big]^3\times C^3(J)\big)$ by continuous dependence w.r.t. the initial condition, since this latter satisfies $v_0\in \big[C(J)\big]^3\times C^3(J)$ and ---the considered restriction of--- $F$ is of class $C^\infty$.
\end{proof}

The infinitesimal generator of the model, or of \eqref{CompactRenormalizedDeterministicPDE}, is given by 
\begin{equation}\label{DeterministicSpatialGenerator}
\disp \cA\varphi(u)=\langle D\varphi(u),\tilde{A}u+F(u)\rangle,
\end{equation} 
on the domain $\disp C_b^1\left(\big[C(J)\big]^3\times C^3(J)\right)$.  
The associated debit function is the vector field in $\bold C(J)$ defined by \[ \psi(u)=\tilde{A}u+F(u). \] 
More precisely, $u\mapsto F(u)$ is the debit related to the fluctuations of $v$ due to events that are spatially homogeneous, while $u\mapsto \tilde{A}u$ is the debit related to the fluctuations created by spatial motions.
A much more specific decomposition is possible for $\psi$. It consists in writing $\psi=(\psi_S,\psi_I,\psi_R,\psi_B)$, where $\psi_S(u)=F_S(u)$, $\psi_I(u)=F_I(u)$, $\psi_R(u)=F_R(u)$ and $\psi_B(u)=Au_B+F_B(u)$ are the debits related to the fluctuations of $S$, $I$, $R$, and $B$ respectively, defined for all $u=(u_S,u_I,u_R,u_B)\in \big[C(J)\big]^3\times C^3(J)$.

\subsection{Stochastic spatial model}

We now introduce the stochastic counterpart of $v=(S,I,R,B)$. The preceding subdivision of the spatial domain is considered. In order to model event randomness, as it is usual we use Poisson processes as random clock to capture events. For each site $i\in\{1,\cdots,N\}$, proportions for compartments are denoted $u_{S,i}^N$, $u_{I,i}^N$, $u_{R,i}^N$, $u_{B,i}^N$. Recall that the parameter of renormalization is $K$ (resp. $H=\bH(0)/N$)  for bateria (resp. humans), so that 
\[ u_{B,i}^N=\frac{\text{number of bacteria on i}}{K},\hspace{0.5cm} u_{S,i}^N=\frac{\text{number of susceptible on i}}{H},\hspace{0.5cm}\cdots \] 

We have $\disp u_{S,i}^N,u_{I,i}^N,u_{R,i}^N\in H^{-1}\N$ and $\disp u_{B,i}^N\in K^{-1}\N$, where \[ H^{-1}\N=\left\{H^{-1}n,n\in\N\right\}\hspace{0.5cm}\text{and}\hspace{0.5cm}K^{-1}\N=\left\{K^{-1}n,n\in\N\right\}. \]
As previously, we are omitting the parameters $H$ and $K$ in our notation. \\

A pointwise modeling over the whole spatial domain is achieved for susceptible through the step function \[ u_S^N(t,x)=\sum_{i=1}^Nu_{S,i}^N(t)\1_i(x), \hspace{0.5cm}t\geq0,\hspace{0.5cm}x\in J, \] where $\1_i(\cdot)=\1_{J_i}(\cdot)$ is the indicator function of the $i$-th site $J_i$. For the compartments (infected, recovered and bacteria), $u_S^N(t,x)$, $u_R^N(t,x)$ and $u_B^N(t,x)$ are defined in a similar way as $u_S^N(t,x)$. Using the standard identification $u_S^N(t):=u_S^N(t,\cdot)$, $u_I^N(t):=u_I^N(t,\cdot)$, ... we have $\big(u_S^N(t), u_I^N(t), u_R^N(t), u_B^N(t)\big)\in \bold L^2(J)$. 
Finally, a global description including all the compartments is given by 

\[ u^N(t)=\big(u_S^N(t),u_I^N(t),u_R^N(t),u_B^N(t)\big), \] \vspace{0cm}

\noindent which is our stochastic model. 
In absence of precision, we consider the natural completed filtration $\big\{\cF_t^N,t\geq0\big\}$, where $\cF_t^N=\cF_t^{N,H,K}$ is the completion of the $\sigma$-algebra $\sigma\left(\big\{u^N(s):s\leq t\big\}\right)$ with null probability measure sets. 

Let $\H^N:=\H^N(J)$ be the subspace of $L^2(J)$ which consists of 
real valued step functions that are constant on each site $J_i$. Introduce the canonical projection
\begin{equation}\label{CanonicalProjection}
\begin{tabular}{lll}
$\displaystyle P_N$ & $:$ & $L^2(J) \longrightarrow \mathds{H}^N$\vspace{0cm}\\
 & & $\displaystyle f\longmapsto P_Nf=\sum_{i=1}^Nf_i\mathds{1}_i,\hspace{0.4cm}\text{where}\hspace{0.4cm} f_i:=N\int_{J_i}f(x)dx$.
\end{tabular}
\end{equation}
Then, $\big(\H^N,\langle\cdot,\cdot\rangle_2\big)$ is an $N$-dimensional Hilbert space. An orthonormal basis of it is $\{e_i:=\sqrt{N}\1_i,1\leq i\leq N\}$. Its inner product is the restriction on $\H^N$ of the inner product of $L^2(J)$, that is $\langle f,g\rangle_2=\frac{1}{N}\sum_{i=1}^Nf_ig_i$. Setting $\H=\cup_{N\geq1}\H^N$, it is not difficult to see that $(\H,\langle\cdot,\cdot\rangle_2)$ is dense in $L^2(J)$. Indeed, every function $f\in L^2(J)$ can be approximated by the sequence $(P_Nf)_N\subset\H$. 

Furthermore, $\H \subset C_p(J)$ and $\Vert P_Nf-f\Vert_\infty\rightarrow 0$ for all $f\in C_p(J)$. As our goal is to perform approximations in the supremum norm, we systematically consider $(\H^N,\Vert\cdot\Vert_\infty)$, unless we specify another topology.

\medskip
In addition, we define $\tilde{P}_N$ on $\bold L^2(J)$ by $\big(u^{1},\cdots,u^{4})\mapsto\tilde{P}_Nu=\big(P_Nu^{1},\cdots,P_Nu^{4})$, and introduce the notation $\bold H^N :=\H^N\times \H^N\times \H^N\times\H^N =: \bold H^N(J)$. According to our preceding discussion, we consider the Banach space $\big(\bold H^N, \Vert\cdot\Vert_{\bold C_p(J)}\big)$ in what follows, and recall that $\Vert\cdot\Vert_{\bold C_p(J)}=\Vert\cdot\Vert_{\bold C(J)}$.

\medskip
Now, for all $t\geq0$, the $i$-th coordinate of $u^N(t)$ with respect to the projection $\tilde{P}_N$ is 
\[ u_i^N(t) \quad = \quad N\int_{J_i}u^N(t,x)dx \quad =\quad \big(u_{S,i}^N(t),u_{I,i}^N(t),u_{R,i}^N(t),u_{B,i}^N(t)\big), \] 
so that $\tilde{P}_Nu^N(t)=u^N(t)$, and  \[ u^N(t,x)=\sum_{i=1}^Nu_{i}^N(t)\1_i(x),\hspace{0.5cm}\forall x\in J. \] 
Therefore, $u^N=\big\{u^N(t),t\geq0\}$ is an $\bold H^N-$valued jump Markov process. Its infinitesimal generator is given on the domain $C_b\big(\bold H^N\big)$ by 
\begin{equation}\label{StochasticSpatialGenerator}
\begin{array}{l}
\disp \cA^N\varphi(u)=\cA_S^N\varphi(u)+\cA_{SI}^N\varphi(u)+\cA_I^N\varphi(u)\vspace{0.2cm}\\
\hspace{3cm}+\cA_{IR}^N\varphi(u)+\cA_R^N\varphi(u)+\cA_{RS}^N\varphi(u)+\cA_B^N\varphi(u),
\end{array}
\end{equation}

\noindent where $u=(u_S,u_I,u_R,u_B)$,
\begin{equation}\label{StochasticSpatialGenerator_S}
\begin{array}{l}
\disp \cA_S^N\varphi(u_S,u_I,u_R,u_B)\\
\disp \hspace{1cm} = \sum_{i=1}^N\left\{\left[\varphi\left(u_S+\frac{1}{H}\1_i,u_I,u_R,u_B\right)-\varphi(u_S,u_I,u_R,u_B)\right]H\mu u_{S,i}\right.\\
\disp \hspace{2.5cm} +\left[\varphi\left(u_S+\frac{1}{H}\1_i,u_I,u_R,u_B\right)-\varphi(u_S,u_I,u_R,u_B)\right]H\mu u_{I,i}\\
\disp \hspace{2.5cm} +\left[\varphi\left(u_S+\frac{1}{H}\1_i,u_I,u_R,u_B\right)-\varphi(u_S,u_I,u_R,u_B)\right]H\mu u_{R,i}\\
\disp \hspace{2.5cm} +\left.\left[\varphi\left(u_S-\frac{1}{H}\1_i,u_I,u_R,u_B\right)-\varphi(u_S,u_I,u_R,u_B)\right]H\mu u_{S,i}\right\},
\end{array}
\end{equation}

\begin{equation}\label{StochasticSpatialGenerator_SI}
\begin{array}{l}
\disp \cA_{SI}^N\varphi(u_S,u_I,u_R,u_B)\\
\disp \hspace{1cm} = \sum_{i=1}^N\left[\varphi\left(u_S-\frac{1}{H}\1_i,u_I+\frac{1}{H}\1_i,u_R,u_B\right)-\varphi(u_S,u_I,u_R,u_B)\right]H\beta\frac{u_{B,i}}{1+u_{B,i}}u_{S,i},
\end{array}
\end{equation}

\begin{equation}\label{StochasticSpatialGenerator_I}
\begin{array}{l}
\disp \cA_I^N\varphi(u_S,u_I,u_R,u_B)\\
\disp \hspace{1cm} = \sum_{i=1}^N\left\{\left[\varphi\left(u_S,u_I-\frac{1}{H}\1_i,u_R,u_B\right)-\varphi(u_S,u_I,u_R,u_B)\right]H\mu u_{I,i}\right.\\
\disp \hspace{2.5cm} +\left.\left[\varphi\left(u_S,u_I-\frac{1}{H}\1_i,u_R,u_B\right)-\varphi(u_S,u_I,u_R,u_B)\right]H\alpha u_{I,i}\right\},\\
\end{array}
\end{equation}

\begin{equation}\label{StochasticSpatialGenerator_IR}
\begin{array}{l}
\disp \cA_{IR}^N\varphi(u_S,u_I,u_R,u_B)\\
\disp \hspace{1cm} = \sum_{i=1}^N\left[\varphi\left(u_S,u_I-\frac{1}{H}\1_i,u_R+\frac{1}{H}\1_i,u_B\right)-\varphi(u_S,u_I,u_R,u_B)\right]H\gamma u_{I,i},\\
\end{array}
\end{equation}

\begin{equation}\label{StochasticSpatialGenerator_R}
\begin{array}{l}
\disp \cA_R^N\varphi(u_S,u_I,u_R,u_B)\\
\disp \hspace{1cm} = \sum_{i=1}^N\left[\varphi\left(u_S,u_I,u_R-\frac{1}{H}\1_i,u_B\right)-\varphi(u_S,u_I,u_R,u_B)\right]H\mu u_{R,i},\\
\end{array}
\end{equation}

\begin{equation}\label{StochasticSpatialGenerator_RS}
\begin{array}{l}
\disp \cA_{RS}^N\varphi(u_S,u_I,u_R,u_B)\\
\disp \hspace{1cm} = \sum_{i=1}^N\left[\varphi\left(u_S+\frac{1}{H}\1_i,u_I,u_R-\frac{1}{H}\1_i,u_B\right)-\varphi(u_S,u_I,u_R,u_B)\right]H\rho u_{R,i},\\
\end{array}
\end{equation}
and finally
\begin{equation}\label{StochasticSpatialGenerator_B}
\begin{array}{l}
\disp \cA_B^N\varphi(u_S,u_I,u_R,u_B)\\
\disp \hspace{0.7cm} = \sum_{i=1}^N\left\{\left[\varphi\left(u_S,u_I,u_R,u_B-\frac{1}{K}\1_i\right)-\varphi(u_S,u_I,u_R,u_B)\right]K\mu_B u_{B,i}\right.\\
\disp \hspace{2cm} +\left[\varphi\left(u_S,u_I,u_R,u_B+\frac{1}{K}\1_i\right)-\varphi(u_S,u_I,u_R,u_B)\right]\frac{Hp}{W} u_{I,i}\\
\disp \hspace{2cm} +\left[\varphi\left(u_S,u_I,u_R,u_B+\frac{\1_{i+1}-\1_i}{K}\right)-\varphi(u_S,u_I,u_R,u_B)\right]K\ell\cP_{out} u_{B,i}\\
\disp \hspace{2cm} +\left. \left[\varphi\left(u_S,u_I,u_R,u_B+\frac{\1_{i-1}-\1_i}{K}\right)-\varphi(u_S,u_I,u_R,u_B)\right]K\ell\cP_{in} u_{B,i}\right\}.\\
\end{array}
\end{equation}

The generator $\cA^N$ can be extended to a generator $\bar{\cA}^N$ on $C_b\big(\bold L^2(J)\big)$, using the projection $\tilde{P}_N$ as follows:\[ \bar{\cA}^N\varphi(u):=\cA^N\varphi(\tilde{P}_Nu), \] for $u\in\bold L^2(J)$. We do not distinguish between the two generators and use the notation $\cA^N$ for both.

\medskip
By construction, when started at a positive state $u^N(0)\in \big(H^{-1}\N\big)\times K^{-1}\N$, the process $u^N$ lies in the same set, and thus keeps positive values. We assume that our process starts from such a state in the sequel.\\

From preceding discussions, we know that the debit function of $u^N$ is the vector field $\psi^N=\big(\psi_S^N,\psi_I^N,\psi_R^N,\psi_B^N\big)$ in $\bold H^N$, such that:
\begin{equation}\label{StochasticSpatialDebit_S}
\begin{array}{ll}
\disp \psi_S^N(u_S,u_I,u_R,u_B) & \disp =\sum_{i=1}^N\left[\mu u_{I,i}+(\mu+\rho)u_{R,i}-\beta\frac{u_{B,i}}{1+u_{B,i}}u_{S,i}\right]\1_i \vspace{0.2cm}\\
\disp & \disp =\mu u_I+ (\mu+\rho)u_R-\beta\frac{u_B}{1+u_B}u_S \hspace{0.5cm} = \hspace{0.5cm} F_S(u_S,u_I,u_R,u_B)
\end{array}
\end{equation}
is the debit related to susceptible,
\begin{equation}\label{StochasticSpatialDebit_I}
\begin{array}{ll}
\disp \psi_I^N(u_S,u_I,u_R,u_B) & \disp =\sum_{i=1}^N\left[\beta\frac{u_{B,i}}{1+u_{B,i}}u_{S,i}-(\mu +\alpha +\gamma)u_{I,i}\right]\1_i\vspace{0.2cm}\\
& \disp =\beta\frac{u_B}{1+u_B}u_S-(\mu+\alpha+\gamma)u_I\hspace{0.5cm}  =  \hspace{0.5cm} F_I(u_S,u_I,u_R,u_B)
\end{array}
\end{equation}
is the debit related to infected,
\begin{equation}\label{StochasticSpatialDebit_R}
\begin{array}{ll}
\disp \psi_R^N(u_S,u_I,u_R,u_B) & \disp  =\sum_{i=1}^N\big[\gamma u_{I,i}-(\mu+\rho) u_{R,i}\big]\1_i \vspace{0.2cm}\\
 & \disp = \gamma u_I-(\mu+\rho)u_R \hspace{0.5cm} = \hspace{0.5cm} F_R(u_S,u_I,u_R,u_B) 
\end{array}
\end{equation}
is the debit related to recovered, and lastly,
\begin{equation}\label{StochasticSpatialDebit_B}
\begin{array}{l}
\disp \psi_B^N(u_S,u_I,u_R,u_B)\vspace{0.1cm}\\
\disp \hspace{1cm}=\sum_{i=1}^N\left[-\mu_Bu_{B,i}\1_i+\frac{H p}{KW}u_{I,i}\1_i+(\1_{i+1}-\1_i)\ell\cP_{out}u_{B,i}+(\1_{i-1}-\1_i)\ell\cP_{in}u_{B,i}\right]\vspace{0cm}\\
\disp \hspace{1cm} = \sum_{i=1}^N\left[-\mu_Bu_{B,i}+\frac{H p}{KW}u_{I,i}+\ell(u_{B,i-1}-u_{B,i})\cP_{out}+\ell(u_{B,i+1}-u_{B,i})\cP_{in}\right]\1_i\\
\disp \hspace{1cm}=\sum_{i=1}^N\left\{-\mu_Bu_{B,i}+\frac{H p}{KW}u_{I,i}+\frac{\ell}{2N^2}\big[N^2(u_{B,i-1}-2u_{B,i}+u_{B,i+1})\big]\right.\\
\disp \hspace{3cm}\left.-\frac{\ell}{N}(\cP_{out}-\cP_{in})\left[\frac{N}{2}(u_{B,i+1}-u_{B,i-1})\right]\right\}\vspace{0.1cm}\\
\disp \hspace{1cm} = -\mu_Bu_{B}+\frac{H p}{KW}u_{I}+(-\nu\nabla_N +\fD\Delta_N) u_B \hspace{0.5cm} = \hspace{0.5cm} F_B(u_S,u_I,u_R,u_B)+A_Nu_B
\end{array}
\end{equation}
is the debit related to bacteria. The second equality of \eqref{StochasticSpatialDebit_B} is obtained using a change of index and periodicity, and the next one follows from the identity $\cP_{in}+\cP_{out}=1$. The coefficients $\nu=\ell\cdot b/N$ and $\fD=\ell/2N^2$ are the advection and diffusion coefficients introduced in \cref{ConstantReservoirOnSitesAndConvergenceConditions}. The operator $A_N=-\nu\nabla_N+\fD\Delta_N$ is a discretization of the operator $A$ 
defined by \eqref{CompactRenormalizedDeterministicPDE}. 
Here, $\nabla_N$ and $\Delta_N$ are respectively the centered, discrete, gradient and Laplace operators, defined on $L^2(J)$. The former is given by \[ \nabla_Nf(x)=\frac{N}{2}\left[f\left(x+\frac{1}{N}\right)-f\left(x-\frac{1}{N}\right)\right]. \] 
Concerning the latter, we first introduce the uncentered discrete gradients \[ \nabla_N^+f(x)=N\left[f\left(x+\frac{1}{N}\right)-f(x)\right]\hspace{0.5cm}\text{and}\hspace{0.5cm}\nabla_N^-f(x)=N\left[f(x)-f\left(x-\frac{1}{N}\right)\right] .\] Then, we define the centered discrete Laplace by 
\begin{align*}
\Delta_Nf(x) & = \nabla_N^+\nabla_N^-f(x)\\
& = \nabla_N^-\nabla_N^+f(x) = N^2\left[f\left(x+\frac{1}{N}\right)-2f(x)+f\left(x-\frac{1}{N}\right)\right].
\end{align*} 
If $f\in\H^N$ in particular, then periodicity yields \[ \nabla_Nf=\frac{N}{2}\sum_{i=1}^N\big(f_{i+1}-f_{i-1}\big)\1_i\hspace{0.5cm}\text{and}\hspace{0.5cm}\Delta_Nf=N^2\sum_{i=1}^N\big(f_{i+1}-2f_i+f_{i-1}\big)\1_i. \] \par 

Hence, a compact formulation of the debit is \[ \psi^N(u)=\tilde{A}_Nu+F(u) \] for all $u\in \bold H^N$, where $F$ is given by \eqref{CompactRenormalizedDeterministicPDE}, and $\tilde{A}_N=\text{diag}(0,0,0,A_N)$ is a diagonal matrix operator on $\bold H^N$. This latter operator is a dicretization of the operator $\tilde{A}$ given by \eqref{CompactRenormalizedDeterministicPDE}. \par 

\medskip
The debit $\psi^N$ can also be extended to $\bold L^2(J)$, by $u\mapsto\psi^N\big(\tilde{P}_Nu\big)$.\\

We are interested in the asymptotic behaviour of our stochastic process $u^N$, as $N,H,K\rightarrow\infty$. In an $L^2(J)-$framework, the generator $\cA^N$ given by \eqref{StochasticSpatialGenerator} formally converges to the generator $\cA$ given by \eqref{DeterministicSpatialGenerator}, under the additional assumptions $K^{-1}N^2\rightarrow 0$ and $H/K>0$ is kept constant. We notice that the condition $K^{-1}N^2\rightarrow 0$ is equivalent to $H^{-1}N^2\rightarrow 0$, as soon as $H/K\rightarrow>0$ remains constant. This formal argument strongly suggests the convergence of the stochastic model to a corresponding deterministic model. We rigorously prove this below, in the supremum norm topology, and replacing the strong condition $K^{-1}N^2\rightarrow\infty$ by the much weaker condition $K^{-1}\log N\rightarrow\infty$.

\section{The law of large numbers}

We state and prove our main result.

\begin{theorem}\label{TheLLNForSpatialCholera}
Consider a sequence $u^N=\big(u_S^N,u_I^N,u_R^N,u_B^N\big)$ of Markov processes starting at $u^N(0)=\big(u_S^N(0),u_I^N(0),u_R^N(0),u_B^N(0)\big)$, with the infinitesimal generators $\cA^N$ given by \eqref{StochasticSpatialGenerator}. Assume that: \vspace{0.2cm}\par 
(i) $N,H,K\longrightarrow\infty$ in such a way that $H/K$ remains constant and non-negative, and $K^{-1}\log N\rightarrow0$ or, equivalently, $H^{-1}\log N\rightarrow0$.\par 
(ii) The assumptions of \cref{Well-posednessOfCompactRenormalizedDeterministicPDE} hold and $v$ is the solution of \eqref{CompactRenormalizedDeterministicPDE}. \par  
(iii) $u^N(0)\in \big(H^{-1}\N\big)^3\times K^{-1}\N$ and $\Vert u^N(0)-v(0)\Vert_{\bold C(J)}\longrightarrow0$ in probability.\vspace{0.2cm}\par 
\noindent Then, for all $T>0$, \[ \sup_{t\leq T}\Vert u^N(t)-v(t)\Vert_{\bold C(J)}\longrightarrow0\hspace{0.3cm}\text{in probability}. \]
\end{theorem}

\noindent \textbf{\textit{Proof.}} 
Fix $T>0$. We want to prove that for all $\epsilon>0$, \[ \P\left\{\sup_{t\leq T}\Vert u^N(t)-v(t)\Vert_E>\epsilon\right\}\rightarrow0. \] 

Instead of working directly with $v$, we consider a discrete version $v^N=\big(v_S^N,v_I^N,v_R^N,v_B^N\big)$ of it, defined by the ODE:
\begin{equation}\label{DiscreteVersionOfPDE}
\left\{
\begin{array}{l}
\disp \frac{dv^N(t)}{dt}=\tilde{A}_Nv^N(t)+F\big(v^N(t)\big)\vspace{0.1cm}\\
\disp v^N(0)=\tilde{P}_Nv_0.
\end{array}
\right.
\end{equation}
Let $T_N(t)=e^{A_Nt}$ be the semigroup of $A_N$ on $\big(\H^N,\Vert\cdot\Vert_\infty\big)$, and let $\tilde{T}_N(t)=e^{\tilde{A}_Nt}$ be the semigroup of $\tilde{A}_N$ on $\big(\bold H^N,\Vert\cdot\Vert_{\bold C(J)}\big)$. These are clearly 
bounded semigroups. Let $c_2>0$ be a real such that $\Vert T_N(t)\Vert\leq c_2$ ---the operator norm---. We  claim the following.

\begin{lemma}\label{DiscreteApproximationOfTheLimit}
The initial condition problem \eqref{DiscreteVersionOfPDE} has a unique global mild solution $v^N\in C\big(\R_+;\bold H^N\big)$, satisfying $v^N(t)\geq0$ and
\begin{equation}\label{MildRepresentationOfTheDiscreteVersion}
\disp v^N(t)=\tilde{T}_N(t)\tilde{P}_Nv_0+\int_0^t\tilde{T}_N(t-s)F\big(v^N(s)\big)ds 
\end{equation}
for all $t\geq0$. Moreover, for every $T>0$, 
\begin{equation}\label{BoundOfTheDiscreteVersion}
\sup_{t\leq T}\Vert v^N(t)\Vert_{\bold C(J)}\leq c_T=c_0c_2\text{e}^{c_2MT},
\end{equation}
and
\begin{equation}\label{ConvergenceOfTheDiscreteVersion}
\disp \lim_{N\rightarrow\infty}\sup_{t\in[0,T]}\Vert v^N(t)-v(t)\Vert_{\bold C(J)}=0. 
\end{equation}
\end{lemma}

The proof is posponed to \Cref{ProofOfDiscreteApproximationOfTheLimit}.
In view of \cref{DiscreteApproximationOfTheLimit}, we may work with $v^N$ rather that $v$. In the rest of the article, $c_T$ is a generic constant depending on $c_0$, $c_2$, $T$ and $M$. \par 

\medskip
The rest of the proof is divided in two steps. We first truncate the stochastic model and replace it by its truncation. Then, we conclude with a Gronwall argument. \\

\noindent \textbf{\textit{\underline{Step 1: TRUNCATION.}}}
Set \[ \tau=\tau(N,\epsilon_0)=\inf\big\{t\geq0:\Vert u^N(t)-v^N(t)\Vert_{\bold C(J)}>\epsilon_0\big\}, \] 

\noindent for fixed $\epsilon_0\in (0,1)$, and define $\bar{u}^N=\big(\bar{u}_S^N,\bar{u}_I^N,\bar{u}_R^N,\bar{u}_B^N\big)$ by 
\begin{equation}\label{TruncatedStochasticProcess}
\left\{
\begin{array}{ll}
\disp \bar{u}^N(t)=u^N(t) & \disp \text{for } 0\leq t\leq \tau\leq \infty,\vspace{0.1cm}\\
\disp \bar{u}^N(t)=u^N(\tau)+\int_{\tau}^t\left(\tilde{A}_N\bar{u}^N(s)+F\big(\bar{u}^N(s)\big)\right)ds & \disp \text{for } \tau<t<\infty.
\end{array}
\right.
\end{equation}

\noindent By definition, $\tau$ is a stopping time such that
\begin{align*}
\P\left\{\sup_{t\leq T}\Vert u^N(t)-v^N(t)\Vert_{\bold C(J)}>\epsilon_0\right\} & \leq \P\left\{\sup_{t\leq T}\Vert u^N(t\wedge\tau)-v^N(t\wedge\tau)\Vert_{\bold C(J)}>\epsilon_0\right\}\\
& \leq \P\left\{\sup_{t\leq T}\Vert \bar{u}^N(t)-v^N(t)\Vert_{\bold C(J)}>\epsilon_0\right\}.
\end{align*}
Therefore, we may work with $\bar{u}^N$ instead of $u^N$.

\medskip
\noindent \textbf{\textit{Boundedness and Lipschitz debits.}} 
The truncated process $\bar{u}^N$ has the same dynamic as $u^N$ until time $\tau$, and follows the flow of an ODE after time $\tau$ if $\tau<\infty$. We then can derive that, 
\begin{equation}\label{BoundednessOfTheTruncatedProcess}
0\leq \bar{u}^N(t,x)\leq c_T, \quad\forall t\leq T,\forall x\in J.
\end{equation} 
Indeed, $u^N(0)\geq0$ and from \cref{Well-posednessOfCompactRenormalizedDeterministicPDE}, we know that $0\leq v(t,x)\leq c_T=c_0e^{MT}$ for all $(t,x)\in [0,T]\times J$. 
Since we are assuming $\Vert u^N(0)-v(0)\Vert_{\bold C(J)}\rightarrow0$ in probability, we may, by conditioning on $\Vert u^N(0)\Vert_{\bold C(J)}<c_T+1$ if necessary, assume without loss of generality that \[ 0\leq u^N(0,x)<c_T+1 \hspace{0.2cm} \forall x\in J,\hspace{0.1cm}\forall N. \]
From 
\cref{DiscreteApproximationOfTheLimit}, we have 
$\sup_{t\leq T}\Vert v^N(t)\Vert_{\bold C(J)}\leq c_T+\frac{1}{2}$.
Therefore, by definition of $\tau$, 
\[ \Vert u^N(t\wedge\tau,x)\Vert_{\bold C(J)}\leq c_T+1,\hspace{0.2cm}\text{for}\hspace{0.2cm}0\leq t\leq \tau \] 
for $\epsilon_0<\frac{1}{2}$.
Now, if $\tau<T<\infty$, for all $t\in(\tau,T]$, we have from the variation of constant 
\[ \bar u^N(t)=\tilde T_N(t)u^N(\tau)+\int_{\tau}^t\tilde T_N(t-s)F\big(\bar u^N(s)\big)ds. \]
Then,
\eqref{ReactionDebitBoundProperties} (iii) yields
\[ \Vert \bar{u}^N(t)\Vert_{\bold C(J)} \leq c_2\Vert \bar{u}^N(\tau)\Vert_{\bold C(J)}+c_2M\int_\tau^t\Vert \bar{u}^N(s)\Vert_{\bold C(J)}ds \leq c_2(c_T+1)\text{e}^{c_2MT}=:c_T, \]
using Gronwall lemma. $\square$

As a result, we consider the restriction of the function $F$ to the compact set $[0,c_T]^4$ of $\R^4$, and consider, in the following, that $F$ is bounded by a constant $M_F(c_T)>0$. Furthermore, $F$ is globally Lipschitz on that compact set and we let $L_F(c_T)$ be a consequent Lipschitz constant.

\medskip
\noindent \textbf{\textit{Accompanying martingales.}} Let us introduce some useful notions and notation.
First, for every node $i=1,\cdots,N$ and all $t\geq0$, define the jump \[ \delta u_i^N(t):=u_i^N(t)-u_i^N(t^-) \] of $u_i^N$ at time $t$. Then denote by $\big|\delta u_i^N(t)\big|$ the amplitude of that jump, where $|(y_1,\cdots,y_4)|=|y_1|+\cdots+|y_4|$. 
Next, we define the \textit{square amplitude} $|\psi_{S}^N|_i^2$ (resp. $|\psi_{B}^N|_i^2$) of $\psi_{S,i}^N$ (resp. $\psi_{B,i}^N$), as the debit function of the process 
\[ \left(H\sum_{s\leq t}|\delta S_i^N(s)|^2\right)_{t\geq0}\hspace{1cm}\left[\text{resp.}\hspace{0.2cm}\left(K\sum_{s\leq t}|\delta B_i^N(s)|^2\right)_{t\geq0}\right]. \] 
We similarly define the \textit{square amplitudes} related to the compartments of infected and recovered, accordingly.
We have: 
\[
\begin{array}{ll}
\bullet & \disp |\psi_{S}^N|_i^2(u) = |F_S|_i^2(u) = 2\mu u_{S,i} + \mu u_{I,i} + (\mu+\rho)u_{R,i} + \beta\frac{u_{B,i}}{1+u_{B,i}}u_{S,i} \vspace{0cm}\\
\bullet & \disp |\psi_{I}^N|_i^2(u) = |F_I|_i^2(u) = \beta\frac{u_{B,i}}{1+u_{B,i}}u_{S,i} + (\mu+\alpha+\gamma)u_{I,i} \vspace{0.1cm}\\
\bullet & \disp |\psi_{R}^N|_i^2(u) = |F_R|_i^2(u) = \gamma u_{I,i} + (\mu+\rho)u_{R,i} \vspace{0.3cm}\\
\bullet & \disp |\psi_{B}^N|_i^2(u) = |A_N|_i^2(u) + |F_B|_i^2(u), 
\end{array}
\] 
where 
\[ |F_B|_i^2(u)=\mu u_{B,i} + r\frac{p}{W}u_{I,i} \quad  and \quad  |A_N|_i^2(u)=\ell(\cP_{in}u_{B,i+1}+u_{B,i}+\cP_{out}u_{B,i-1}). \] 

\indent Also, we set $\disp |\psi^N|_i^2(u)=\left(|\psi_S^N|_i^2+|\psi_I^N|_i^2+|\psi_R^N|_i^2+|\psi_B^N|_i^2\right)(u)$.\par 

Therewith, since spatial correlations induce simultaneous jumps on the nodes $i$ and $i\pm1$, the process \[ \left(K\sum_{s\leq t}\big[\delta B_i^N(s)\big]\big[\delta B_{i\pm1}^N(s)\big]\right)_{t\geq0}, \] with crossed products, is also of interest. Its debit $|\psi_B^N|_i|\psi_B^N|_{i\pm1}$ is given by 
\[ 
\begin{array}{ll}
\bullet & \disp |\psi_B^N|_i|\psi_B^N|_{i+1}(u)=-\ell(\cP_{out}u_{B,i}+\cP_{in}u_{B,i+1}) \vspace{0.2cm}\\
\bullet & \disp |\psi_B^N|_i|\psi_B^N|_{i-1}(u)=-\ell(\cP_{in}u_{B,i}+\cP_{out}u_{B,i-1}). 
\end{array}
\]

In addition, we introduce at last, the \textit{square amplitude function} $|\psi^N|^2$ associated with the debit $\psi^N$. It is the $\bold H^N-$valued function defined by \[ |\psi^N|^2(u)=\sum_{i=1}^N|\psi^N|_i^2(u)\1_i. \] 
The \textit{square amplitude function} can be easily derived for each of the specific debit functions introduced above. 
Concerning $\psi_S^N=F_S$ for instance, it is given by \[ |\psi_S^N|^2(u)(x)=|F_S|^2(u)(x)=\sum_{i=1}^N|F_S|_i^2(u)\1_i(x). \] 
The others are derived accordingly.\\

Now, we move on to the so-called accompanying martingales. As the stopping time $\tau$ satisfies 
\begin{align*}
\sup_{t\leq T}\Vert u^N(t\wedge\tau)\Vert_{\bold C(J)} & \leq \sup_{t\leq T}\Vert \bar{u}^N(t)\Vert_{\bold C(J)}<\infty \leq c_T,
\end{align*}  
various types of martingales can be pointed out, that are associated to the stopped Markov process $u^N(t\wedge\tau)$. 
Define the process $Z^N=\big(Z_S^N,Z_I^N,Z_R^N,Z_B^N\big)$ by
\begin{equation}\label{ZeroTypeMartingales}
\disp Z^N(t)=u^N(t)-u^N(0)-\int_0^t\psi^N\big(u^N(s)\big)ds,
\end{equation}  
where \[ Z_S^N(t)=u_S^N(t)-u_S^N(0)-\int_0^t\psi_S^N\big(u^N(s)\big)ds, \] and the other components are defined accordingly. It is well known that $Z^N(t\wedge\tau)$ defines an $\bold H^N-$valued mean zero $\cF_t^N$-martingale, so that $Z_S^N(t\wedge\tau)$, $Z_I^N(t\wedge\tau)$, $Z_R^N(t\wedge\tau)$ and $Z_B^N(t\wedge\tau)$ are $\H^N-$valued mean zero $\cF_t^N$-martingales. We denote these as \textit{Zero-type associated martingales.} In other words, debit functions are martingale parts of the Markov processes they are associated with. This statement is the basis of what follows.

\begin{lemma}\label{FirstTypeOfMartingales}\texttt{(Martingales of type 1)}\vspace{0.1cm}\par  
For every $i=1,\cdots,N$, the following are mean zero $\cF_t^N$-martingales:
\[
\begin{array}{ll}
\textbf{(i)} & \disp \sum_{s\leq t\wedge\tau}\big[\delta u_{S,i}^N(s)\big]^2-\frac{1}{H}\int_0^{t\wedge\tau}|\psi_S^N|_i^2\big(u^N(s)\big)ds,\\
\textbf{(ii)} & \disp \sum_{s\leq t\wedge\tau}\big[\delta u_{I,i}^N(s)\big]^2-\frac{1}{H}\int_0^{t\wedge\tau}|\psi_I^N|_i^2\big(u^N(s)\big)ds,\\
\textbf{(iii)} & \disp \sum_{s\leq t\wedge\tau}\big[\delta u_{R,i}^N(s)\big]^2-\frac{1}{H}\int_0^{t\wedge\tau}|\psi_R^N|_i^2\big(u^N(s)\big)ds,\\
\textbf{(iv)} & \disp \sum_{s\leq t\wedge\tau}\big[\delta u_{B,i}^N(s)\big]^2-\frac{1}{K}\int_0^{t\wedge\tau}|\psi_B^N|_i^2\big(u^N(s)\big)ds,\\
\textbf{(v)} & \disp \sum_{s\leq t\wedge\tau}\big[\delta u_{B,i}^N(s)\big]\big[\delta u_{B,i\pm1}^N(s)\big]-\frac{1}{K}\int_0^{t\wedge\tau}|\psi_B^N|_i|\psi_B^N|_{i\pm1}\big(u^N(s)\big)ds.
\end{array}
\]
\end{lemma}

\begin{lemma}\label{SecondTypeOfMartingales}\texttt{(Martingales of type 2)}\vspace{0.1cm}\par  
Let $f\in \H^N$. For all $i=1,\cdots,N$, the following are mean zero $\cF_t^N$-martingales:
\[
\begin{array}{ll}
\textbf{(i)} & \disp \sum_{s\leq t\wedge\tau}\left[\delta\big\langle Z_S^N(s),f\big\rangle_2\right]^2-\frac{1}{NH}\int_0^{t\wedge\tau}\left\langle |\psi_S^N|^2\big(u^N(s)\big),f^2\right\rangle_2ds \\
\textbf{(ii)} & \disp \sum_{s\leq t\wedge\tau}\left[\delta\big\langle Z_I^N(s),f\big\rangle_2\right]^2-\frac{1}{NH}\int_0^{t\wedge\tau}\left\langle |\psi_I^N|^2\big(u^N(s)\big),f^2\right\rangle_2ds \\
\textbf{(iii)} & \disp \sum_{s\leq t\wedge\tau}\left[\delta\big\langle Z_R^N(s),f\big\rangle_2\right]^2-\frac{1}{NH}\int_0^{t\wedge\tau}\left\langle |\psi_R^N|^2\big(u^N(s)\big),f^2\right\rangle_2ds\vspace{0.2cm}\\
\textbf{(iv)} & \disp \sum_{s\leq t\wedge\tau}\left[\delta\big\langle Z_B^N(s),f\big\rangle_2\right]^2-\frac{1}{NK}\int_0^{t\wedge\tau}\left[\left\langle u_B^N(s),\cD\left(\big(\nabla_N^+f\big)^2\cP_{out}+\big(\nabla_N^-f\big)^2\cP_{in}\right)\right\rangle_2\right. \\
& \disp \hspace{7.5cm} +\left.\left\langle |F_B|^2\big(u^N(s)\big),f^2\right\rangle_2\right]ds .
\end{array} 
\]
\end{lemma}

The prooves of \cref{FirstTypeOfMartingales} and of \cref{SecondTypeOfMartingales} are similar to those of their counterparts in \cite{DebusscheNankep2017}, \cite{Nankep2018} Chapter 2 or in \cite{Blount1992}.\par

\medskip
\noindent\textbf{\textit{Jump estimates.}} By definition, the truncated process satisfies
\begin{equation}\label{IntegralCompactFormOfTruncatedProcess}
\bar{u}^N(t)=u^N(0)+\int_0^t\left(\tilde{A}_N\bar{u}^N(s)+F\big(\bar{u}^N(s)\big)\right)ds+Z^N(t\wedge\tau),
\end{equation}
and its jumps have the bounds:
\begin{equation}\label{BoundsOfJumpsOfTruncatedProcess}
\left\{
\begin{array}{l}
\disp \Vert\delta \bar{u}_S^N(t)\Vert_\infty=\Vert \delta Z_S^N(t\wedge\tau)\Vert_\infty=\Vert \delta u_{S,i}^N(t\wedge\tau)\Vert_\infty= H^{-1}\vspace{0.1cm}\\
\disp \Vert\delta \bar{u}_I^N(t)\Vert_\infty=\Vert \delta Z_I^N(t\wedge\tau)\Vert_\infty=\Vert \delta u_{I,i}^N(t\wedge\tau)\Vert_\infty= H^{-1}\vspace{0.1cm}\\
\disp \Vert\delta \bar{u}_R^N(t)\Vert_\infty=\Vert \delta Z_R^N(t\wedge\tau)\Vert_\infty=\Vert \delta u_{R,i}^N(t\wedge\tau)\Vert_\infty= H^{-1}\vspace{0.1cm}\\
\disp \Vert\delta \bar{u}_B^N(t)\Vert_\infty=\Vert \delta Z_B^N(t\wedge\tau)\Vert_\infty=\Vert \delta u_{B,i}^N(t\wedge\tau)\Vert_\infty= K^{-1}
\end{array}
\right.
\end{equation}
for all $t\geq0$.\\

From now on, we consider the truncated process and write $\bar{u}^N=u^N$ in order to simplify our notation.\\

\noindent \textbf{\textit{\underline{Step 2: A Gronwall-Bellman argument.}}}
We study the difference 
\[ u^N(t)-v^N(t)=\big(u_S^N(t)-v_S^N(t),u_I^N(t)-v_I^N(t),u_R^N(t)-v_R^N(t),u_B^N(t)-v_B^N(t)\big). \]
where we recall $u^N$ is the truncated process. Variation of constant at \eqref{IntegralCompactFormOfTruncatedProcess} yields
\begin{equation}\label{MildCompactFormOfTruncatedProcess}
u^N(t)=\tilde{T}_N(t)u^N(0)+\int_0^t\tilde{T}_N(t-s)F\big(u^N(s)\big)ds+Y^N(t),
\end{equation}
where $Y^N(t)=\int_0^t\tilde{T}_N(t-s)dZ^N(s\wedge\tau)$. It should be noted that $s\mapsto Z^N\big(s\wedge\tau,\frac{i}{N}\big)$ is of bounded variation for $i=1,\cdots,N$, and $\tilde{T}_N$ may be viewed as a $4N\times4N$ matrix-valued function. 
Hence, $Y^N\big(t,\frac{i}{N}\big)$, $1\leq i\leq N$, is defined as a Stieltjes integral.

From \eqref{MildRepresentationOfTheDiscreteVersion} and \eqref{MildCompactFormOfTruncatedProcess},
\begin{align*}
u^N(t)-v^N(t) & =\tilde{T}_N(t)\big(u^N(0)-v^N(0)\big)\\
& \hspace{0.5cm}+\int_0^t\tilde{T}_N(t-s)\left[F\big(u^N(s)\big)-F\big(v^N(s)\big)\right]ds+Y^N(t).
\end{align*}
Since $F$ is Lipschitz and $\tilde{T}_N(t)$ is 
bounded, we get from Gronwall lemma
\begin{align*}
\sup_{t\leq T}\Vert u^N(t)-v^N(t)\Vert_{\bold C(J)}\leq \left(c_2\Vert u^N(0)-v^N(0)\Vert_{\bold C(J)}+\sup_{t\leq T}\Vert Y^N(t)\Vert_{\bold C(J)}\right)\text{e}^{c_2TL_F}.
\end{align*}
By assumption, 
\[ \Vert u^N(0)-v^N(0)\Vert_{\bold C(J)}\leq \Vert u^N(0)-v(0)\Vert_{\bold C(J)}+\Vert \tilde{P}_Nv(0)-v(0)\Vert_{\bold C(J)}\longrightarrow0 \] 
in probability. Therefore, the proof of \cref{TheLLNForSpatialCholera} is completed if we show that 
\begin{equation}\label{MartingalePartConvergenceToZero}
\sup_{t\leq T}\Vert Y^N(t)\Vert_{\bold C(J)}\longrightarrow0\hspace{0.3cm} \text{in probability}.
\end{equation}
As $Y^N=\big(Y_S^N,Y_I^N,Y_R^N,Y_B^N\big)$, it suffices to prove that \[ \sup_{t\leq T}\Vert Y_{index}^N(t)\Vert_{\bold C(J)}\longrightarrow0\hspace{0.2cm} \text{in probability},\hspace{0.3cm}\text{for}\hspace{0.3cm}\small index\normalsize = S,I,R,B, \] where $Y_{index}^N(t)=Z_{index}^N(t\wedge\tau)$ for $\small index \normalsize =S,I,R$, and $Y_B^N(t)=\int_0^tT_N(t-s)dZ_B^N(s\wedge\tau)$.\\

Below, each component is treated at once. 
We first introduce a very useful result. 

\begin{lemma}\label{BlountExpectationEstimate} \textbf{\texttt{(Lemma 4.4, \cite{Blount1992})}}
Let $m(t)$ be a bounded martingale of finite variation defined on $[t_0,t_1]$, with $m(t_0)~=~0$, and satisfying:\vspace{0.1cm}\par 

\textbf{(i)} $m$ is right-continuous with left limits.\par 
\textbf{(ii)} $|\delta m(t)|\leq 1$ for $t_0\leq t\leq t_1$. \par 
\textbf{(iii)} $\displaystyle \sum_{t_0\leq s\leq t}[\delta m(s)]^2-\int_{t_0}^tg(s)ds$ is a mean $0$ martingale with $0\leq g(s)\leq h(s)$, where $h(s)$ is a bounded deterministic function and $g(s)$ is $\displaystyle \mathcal{F}_t^N$-adapted.\par 

Then, \[ \mathds{E}\left[\text{e}^{m(t_1)}\right]\leq \exp\left(\frac{3}{2}\int_{t_0}^{t_1}h(s)ds\right). \]
\end{lemma}

\noindent \textbf{\textit{Components related to Human.}} 
Let us start with $Y_S^N$. 
We want to prove that \[ \P\left\{\sup_{t\leq T}\Vert Y_S^N(t)\Vert_\infty>\epsilon_0\right\}\longrightarrow0. \]
Fix $\bar{t}\in(0,T]$, $i\in\{1,\cdots,N\}$, and, for $0\leq t\leq \bar{t}$, set \[ f:=N\1_i \hspace{1cm}\text{and}\hspace{1cm}\bar{m}_S(t)=\left\langle Z_S^N(t\wedge\tau),f\right\rangle_2. \]
Then, $\bar{m}_S=\big\{\bar{m}_S(t),0\leq t\leq \bar{t}\big\}$ is a mean zero martingale such that 
\[ \bar{m}_S(\bar{t})=Z_{S,i}^N(\bar{t}\wedge\tau)=Y_{S,i}^N(\bar{t}). \] 
From \cref{SecondTypeOfMartingales}, 
\[ \sum_{s\leq t\wedge\tau}[\delta\bar{m}_S(s)]^2-\frac{1}{NH}\int_0^t\left\langle |F_S|^2\big(u^N(s)\big),f^2\right\rangle_2ds \] 
is a mean 0 martingale for $0\leq t\leq\bar{t}$. 
Now, \eqref{BoundsOfJumpsOfTruncatedProcess} yields \[ |\delta\bar{m}_S(t)|\leq \Vert \delta Z_S^N(t\wedge\tau)\Vert_\infty=\Vert \delta u_S^N(t\wedge\tau)\Vert_\infty\leq H^{-1}. \]
Then, for $\theta\in[0,1]$, \[ m_S(t)=\theta H\bar{m}_S(t) \] defines a mean zero \textit{càdlàg} martingale such that $|\delta m_S(t)|\leq 1$ and $[\delta m_S(s)]^2=\theta^2H^2[\delta\bar{m}_S(s)]^2$.   
Thus, from \cref{SecondTypeOfMartingales}, \[ \sum_{s\leq t}[\delta m_S(s)]^2-\int_0^{t\wedge\tau}g_S^N(s)ds \] defines a mean zero \textit{càdlàg} martingale, where
\begin{align*}
g_S^N(s) & =\frac{\theta^2H}{N}\left\langle |F_S|^2\big(u_S^N(s)\big),f^2\right\rangle_2=\theta^2H|F_S|_i^2\big(u_S^N(s)\big)\leq c\theta^2H
\end{align*} 
for $0\leq s\leq t\wedge\tau$. 
Here $c=c(\mu,\rho,\beta,C_T)$, and below, $c$ is considered as a generic constant that depends on $T$. It follows that $0\leq g_S^N(s)\leq h_S^N(s)$, where $h_S^N(s)=c\theta^2H$ satisfies $\int_0^th_S^N(s)ds\leq c\theta^2H$ for $0\leq t\leq \bar{t}\leq T$. 
Therefore, \cref{BlountExpectationEstimate} implies $\E[\text{e}^{m_S(\bar{t})}]\leq \exp(c\theta^2 H)$, and from Markov's inequality,
\begin{align*}
\P\left\{Y_{S,i}^N(\bar{t})>\epsilon_0\right\} & =\P\left\{\bar{m}_{S}(\bar{t})>\epsilon_0\right\}=\P\left\{m_{S}(\bar{t})>\theta H\epsilon_0\right\}=\P\left\{\text{e}^{m_{S}(\bar{t})}>\text{e}^{\theta H\epsilon_0}\right\}\\
& \leq \text{e}^{-\theta H\epsilon_0}\E[\text{e}^{m_S(\bar{t})}]\\
& \leq \exp\left[\theta H(c\theta-\epsilon_0)\right].
\end{align*}
Thus we can choose $\theta$ such that 
\[ \P\left\{Y_{S,i}^N(\bar{t})>\epsilon_0\right\}\leq \text{e}^{-\eta\epsilon_0^2H}, \quad \text{for some} \quad \eta=\eta(T, C_T)>0, \] independently of $N$, $H$, $i$ and $\bar{t}$. 
Indeed, one may solve $c\theta^2-\epsilon_0\theta+\eta\epsilon_0^2\leq 0$ w.r.t. $\theta$. Below, $\eta$ is generic. 
The relation above holds for $\P\big\{-Y_{S,i}^N(\bar{t})>\epsilon_0\big\}$, repeating the argument with the processes $\bar{m}_S$ and $Y_S^N$ replaced by their opposites $-\bar{m}_S$ and $-Y_S^N$.
Therefore, 
\[ \P\left\{\left|Y_{S,i}^N(t)\right|>\epsilon_0\right\}\leq 2\text{e}^{-\eta\epsilon_0^2H},\hspace{0.3cm}\text{for}\hspace{0.3cm}0\leq t\leq T\text{ and } i=1,\cdots,N. \]
Since $\Vert Y_S^N(t)\Vert_\infty=\sup_{i=1,\cdots,N}|Y_{S,i}^N(t)|$ and $Y_S^N(0)=0$,
\begin{equation}\label{BoundInTheSupremumNormOnY_SOf_t_InProbability}
\begin{array}{ll}
\disp \P\left\{\left\Vert Y_{S}^N(t)\right\Vert_\infty>\epsilon_0\right\} & \disp =\P\left\{\exists i=1,\cdots,N:|Y_{S,i}^N(t)|>\epsilon_0\right\}\\
 & \disp \leq \sum_{i=1}^N\P\left\{\left|Y_{S,i}^N(t)\right|>\epsilon_0\right\}\\
 & \disp \leq 2Ne^{-\eta\epsilon_0^2H}
\end{array}
\end{equation}
for $0\leq t\leq T$, $\eta=\eta(C_T)>0$. \par 

Now, we show that \eqref{BoundInTheSupremumNormOnY_SOf_t_InProbability} holds with $\Vert Y_S^N(t)\Vert_\infty$ replaced by $\sup_{t\leq T}\Vert Y_S^N(t)\Vert_\infty$ and $N$ replaced by $N^3$ on the r.h.s. Indeed, we subdivide $[0,T]$ into $N^2$ subintervals denoted $I_n(T)=[\frac{nT}{N^2},\frac{(n+1)T}{N^2}]$, $0\leq n\leq N^2-1$. 
Observing that we can always write 
\[ Y_S^N(t)=\tilde{m}_S(t)+Z_S^N\left(\frac{nT}{N^2}\wedge\tau\right)=\tilde{m}_S(t)+Y_S^N\left(\frac{nT}{N^2}\right), \] 
where $\tilde{m}_S(t)=Z_S^N(t\wedge\tau)-Z_S^N\left(\frac{nT}{N^2}\wedge\tau\right)$ is a mean zero martingale for $t\in I_n(T)$, we get 
\begin{equation}\label{IntermediateBoundUniformOnSmallTimeIntervalForY_S}
\sup_{I_n(T)}\Vert Y_S^N(t)\Vert_\infty\leq \sup_{I_n(T)}\Vert \tilde{m}_S(t)\Vert_\infty+\left\Vert Y_S^N\left(\frac{nT}{N^2}\right)\right\Vert_\infty. 
\end{equation}
We are using the notation $\disp \sup_{I_n(T)}$ for $\disp \sup_{t\in I_n(t)}$.
As previously, we fix $i=1,\cdots,N$, $\theta\in[0,1]$ and set $m_S(t)=\theta H\tilde{m}_S\left(t,\frac{i}{N}\right)$ for $t\in I_n(T)$. 
Thus, for $s\in I_n(T)$, we have $|\delta m_S(s)|\leq 1$, and \cref{SecondTypeOfMartingales} yields 
\[ \sum_{\frac{nT}{N^2}\wedge\tau\leq s\leq t\wedge\tau}[\delta m_S(t)]^2-\int_{\frac{nT}{N^2}\wedge\tau}^{t\wedge\tau}g_S^N(s)ds \] 
is a mean $0$ martingale, with $g_S^N(s)=\theta^2H|F_S|^2\big(u^N(s)\big)\leq \theta^2Hc=h_S^N(s)$. 
Then $0\leq g_S^N(s)\leq h_S^N(s)$ and $\int_{\frac{nT}{N^2}\wedge\tau}^{t\wedge\tau}h_S^N(s)ds \leq \int_{I_n(T)}h_S^N(s)ds \leq c\theta^2H$ since $h_S^N$ is positive. 
Thus, \cref{BlountExpectationEstimate} yields $\E\big[\exp\big\{m_S\big(\frac{(n+1)T}{N^2}\big)\big\}\big]\leq \exp\big(c\theta^2H\big)$, and from Doob's inequalities, 
\begin{align*}
\P\left\{\sup_{I_n(T)}\tilde{m}_S\left(t,\frac{i}{N}\right)>\epsilon_0\right\} & \leq \text{e}^{-\theta H\epsilon_0}\E\left[\exp\left\{m_S\left(\frac{(n+1)T}{N^2}\right)\right\}\right]\\
& \leq \exp\big[\theta H(c\theta-\epsilon_0)]\hspace{1cm}\leq \hspace{1cm} \text{e}^{-\eta\epsilon_0H}
\end{align*}
where $\eta=\eta(C_T)>0$, independently of $N$, $H$ and $i$.  
A suitable $\theta$ has been chosen as previously. 
Also, the same holds for $-\tilde{m}_S\big(t,\frac{i}{N}\big)$. This shows that
\begin{equation}\label{BoundOfTheIntermediateMartingalePartOnSmallTimeIntervalForY_S}
\P\left\{\sup_{I_n(T)}\Vert\tilde{m}_S(t)\Vert_\infty>\epsilon_0\right\}\leq 2N\text{e}^{-\eta\epsilon_0^2H}.
\end{equation}
From \eqref{BoundInTheSupremumNormOnY_SOf_t_InProbability}, \eqref{IntermediateBoundUniformOnSmallTimeIntervalForY_S} and \eqref{BoundOfTheIntermediateMartingalePartOnSmallTimeIntervalForY_S},
\begin{align*}
\P\left\{\sup_{I_n(T)}\Vert Y_S^N(t)\Vert_\infty>\epsilon_0\right\} & \leq \P\left\{\left\Vert Y_S^N\left(\frac{nT}{N^2}\right)\right\Vert_\infty+\sup_{I_n(T)}\Vert \tilde{m}_S(t)\Vert_\infty>\epsilon_0\right\}\\
& \leq  \P\left\{\left\Vert Y_S^N\left(\frac{nT}{N^2}\right)\right\Vert_\infty>\frac{\epsilon_0}{2}\right\}+\P\left\{\sup_{I_n(T)}\Vert \tilde{m}_S(t)\Vert_\infty>\frac{\epsilon_0}{2}\right\}\\
& \leq 2N\text{e}^{-\eta H\epsilon_0^2/2}+2N\text{e}^{-\eta H\epsilon_0^2/2}\\
& \leq 4N\text{e}^{-\eta H\epsilon_0^2}.
\end{align*}
						Hence, 
\begin{align*}
\P\left\{\sup_{t\leq T}\Vert Y_S^N(t)\Vert_\infty>\epsilon_0\right\} & \leq \sum_{n=0}^{N^2-1}\P\left\{\sup_{I_n(T)}\Vert Y_S^N(t)\Vert_\infty>\epsilon_0\right\} \leq 4N^3\text{e}^{-\eta H\epsilon_0^2}
\end{align*}
and it follows that
\begin{align*}
 \P\left\{\sup_{t\leq T}\Vert Y_S^N(t)\Vert_\infty>\epsilon_0\right\} & \leq 4N^3\exp\left(-\eta \epsilon_0^2H\right) \leq \exp\big(\log4+3\log N-\eta H\big).
\end{align*}
The r.h.s. in the second inequality vanishes, since we are assuming $H^{-1}\log N\rightarrow 0$. The expected result for the component related to susceptible is then proved. $\square$\\

Concerning the other human compartment components $Y_I^N$ and $Y_R^N$ related to infected and recovered respectively, we treat them using exactly the same argument as for $Y_S^N$, and we prove that they vanish in probability at the limit. $\square$\\


\noindent \textbf{\textit{Component related to bacteria.}}
\cref{TheLLNForSpatialCholera} is proved if we show that 
\[ \P\left\{\sup_{t\leq T}\Vert Y_B^N(t)\Vert_{\bold C(J)}>\epsilon_0\right\}\longrightarrow0. \]
A similar approach to that of the previous section is used. However, because of the linear part due to the transport, we will need an additional result.
\begin{lemma}\label{Blount_scalar_product_estimate}\textbf{\texttt{(Lemma 4.3, \cite{Blount1992})}}
Set $\displaystyle f=N\mathds{1}_j$. Then \par 
\noindent $\displaystyle \hspace{0cm} \left\langle\left(\nabla_N^+T_N(t)f\right)^2+\left(\nabla_N^-T_N(t)f\right)^2+\left(T_N(t)f\right)^2,1\right\rangle_2\leq h_N(t),\hspace{0cm}\text{ with }\hspace{0cm}\int_0^th_N(s)ds\leq \bar{c}N+t$. 
\end{lemma}

Fix $\bar{t}\in (0,T]$ and $i=1,\cdots,N$. For $0\leq t\leq \bar{t}$, set 
\[ f=N\1_i\hspace{0.5cm}\text{and}\hspace{0.5cm}\bar{m}_B(t)=\left\langle \int_0^tT_N(\bar{t}-s)dZ_B^N(s\wedge\tau),f\right\rangle_2. \]
The process $\bar{m}_B=\big\{\bar{m}_B(t),0\leq t\leq\bar{t}\big\}$ is a mean $0$ martingale such that 
\[ \bar{m}_B(\bar{t})=Y_B^N\left(\bar{t},\frac{i}{N}\right)=Y_{B,i}^N(\bar{t}). \] 

The subsequent result is also needed. It follows from \cite{Blount1992} and \cite{DebusscheNankep2017} or Chapter 2 of \cite{Nankep2018}.

\begin{lemma}\label{ThirdTypeOfMartingales}\texttt{(Martingales of type 3)}\vspace{0.1cm}\par 
\[
\begin{array}{l}
\disp \sum_{s\leq t\wedge\tau}[\delta\bar{m}_B(s)]^2\\
\hspace{1cm} -\frac{1}{N\mu}\int_0^{t\wedge\tau}\left[\left\langle u_B^N(s),\cD\big(\nabla_N^+T_N(\bar{t}-s)f\big)^2\cP_{out}+\cD\big(\nabla_N^-T_N(\bar{t}-s)f\big)^2\cP_{in}\right\rangle_2\right.\\
\hspace{4cm} +\left.\left\langle |F_B|^2\big(u^N(s)\big),\big(T_N(\bar{t}-s)f\big)^2\right\rangle_2\right]ds
\end{array}
\]
defines a mean $0$ càdlàg martingale for $0\leq t\leq\bar{t}$.
\end{lemma}

From \eqref{BoundsOfJumpsOfTruncatedProcess}, 
$|\delta\bar{m}_B(t)|\leq\Vert\delta Z_B^N(t\wedge\tau)\Vert_\infty=\Vert\delta u_B^N(t\wedge\tau)\Vert_\infty\leq K^{-1}$.
Thus, for $\theta\in[0,1]$, the process $m_B$ defined by 
\[ m_B(t)=\theta K\bar{m}_B(t)\hspace{0.5cm}\text{for}\hspace{0.5cm}0\leq t\leq \bar{t} \] is a mean $0$ martingale such that $|\delta m_B(t)|\leq 1$. 
Therefore, from \cref{ThirdTypeOfMartingales} 
\[ \sum_{s\leq t}[\delta m_B(s)]^2-\int_0^{t\wedge\tau}g_B^N(s)ds \] 
is a mean $0$ càdlàg martingale, where for $0\leq s\leq t\wedge\tau$,
\begin{align*}
g_B^N(s) & =\frac{\theta^2K}{N}\left[\left\langle u_B^N(s),\cD\big(\nabla_N^+T_N(\bar{t}-s)f\big)^2\cP_{out}+\cD\big(\nabla_N^-T_N(\bar{t}-s)f\big)^2\cP_{in}\right\rangle_2\right.\\
& \hspace{3cm} +\left.\left\langle |F_B|^2\big(u^N(s)\big),\big(T_N(\bar{t}-s)f\big)^2\right\rangle_2\right]\\
& \leq \frac{\theta^2K}{N}c\left\langle 1,\big(\nabla_N^+T_N(\bar{t}-s)f\big)^2+\big(\nabla_N^-T_N(\bar{t}-s)f\big)^2+\big(T_N(\bar{t}-s)f\big)^2\right\rangle_2\\
& \leq c\theta^2Kh_{C,1}^N(\bar{t}-s)
\end{align*}
where $\int_0^th_{C,1}^N(\bar{t}-s)ds\leq \bar{c}N+t$, thanks to \cref{Blount_scalar_product_estimate}. The constant $c$ is generic and depends on $\cD$, $T$, $c_T$ and $M_F(c_T)$. 
Moreover, from the proof of Lemma 4.3 \cite{Blount1992}, 
one can take $h_{C,1}^N(\bar{t}-s)=1+4\sum_{m>0}\text{e}^{-2\beta_{m,N}(\bar{t}-t)}(\beta_{m,N}+1)$, where $\{-\beta_{m,N}=2N^2\big(\cos(\pi m/N)-1\big)\}_{m\geq0}$ are eigen functions of the discrete Laplace $\Delta_N$. More details can be found in \cite{Blount1987}, Lemma 2.12, p12.
Hence, $g_B^N$ defines an $\cF_t^N$-adapted process such that $0\leq g_B^N(s)\leq h_B^N(s)$, where $h_B^N(s)=\frac{c\theta^2K}{N}h_{C,1}^N(\bar{t}-s)$ is a bounded deterministic function on $[0,\bar{t}]$. As $N\rightarrow\infty$ and $\bar{t}\leq T<\infty$, we may assume $\frac{\bar{t}}{N}\leq 1$ and get $\int_0^{\bar{t}}h_B^N(s)ds\leq c\theta^2K\left(\bar{c}+\frac{\bar{t}}{N}\right)\leq c\theta^2K$.  
\Cref{BlountExpectationEstimate} then implies $\E\big[\text{e}^{m_B(\bar{t})}\big]\leq \exp\big(c\theta^2K\big)$, and by Markov's inequality,
\begin{align*}
\P\left\{Y_{B,i}^N(\bar{t})>\epsilon_0\right\}\leq \text{e}^{-\theta K\epsilon_0}\E\big[\text{e}^{m_B(\bar{t})}\big]\leq \exp\big[\theta K(c\theta-\epsilon_0)\big].
\end{align*}
Thus, we can choose $\theta$ such that 
\[ \P\left\{Y_{B,i}^N(\bar{t})>\epsilon_0\right\}\leq \text{e}^{-\eta\epsilon_0^2K}, \hspace{0.5cm}\text{for}\hspace{0.2cm}\eta=\eta(T,c_T)>0, \] independently of $N$, $K$, $i$ and $\bar{t}$. The constant $\eta=\eta(T,c_T)$ is generic in the following. 
A similar relation is derived, for $-m_B$ and $-Y_B^N$, so that 
\[ \P\left\{|Y_{B,i}^N(t)|>\epsilon_0\right\}\leq 2\text{e}^{-\eta\epsilon_0^2K},\hspace{0.5cm}\text{for}\hspace{0.2cm}0\leq t\leq T\hspace{0.3cm}\text{and}\hspace{0.2cm}i=1,\cdots,N. \]
Since $\Vert Y_B^N(t)\Vert_\infty=\sup_{i=1,\cdots,N}|Y_{B,i}^N(t)|$ and $Y_B^N(0)=0$,
\begin{equation}\label{BoundInTheSupremumNormOnY_BOf_t_InProbability}
\disp \P\left\{\Vert Y_B^N(t)\Vert_\infty>\epsilon_0\right\} \leq 2N\text{e}^{-\eta\epsilon_0^2K},\hspace{0.5cm}\text{for}\hspace{0.3cm}0\leq t\leq T.
\end{equation}

We now show that \eqref{BoundInTheSupremumNormOnY_BOf_t_InProbability} holds when $\Vert Y_B^N(t)\Vert_\infty$ is replaced by $\sup_{t\leq T}\Vert Y_B^N(t)\Vert_\infty$ and $N$ replaced by $N^3$. From Duhamel's formula, $Y_B^N(t)=\int_0^tT_N(t-s)dZ_B^N(s\wedge\tau)$ satisfies the stochastic differential equation (SDE) \[ dY_B^N(t)=A_NY_B^N(t)+dZ_B^N(t\wedge\tau), \] whose integral formulation is 
\begin{equation}\label{IntegralFormulationOfY_BSDE}
Y_B^N(t)=Z_B^N(t\wedge\tau)+\int_0^tA_NY_B^N(s)ds.
\end{equation}
We subdivide $[0,T]$ into $N^2$ subintervals $I_n(T)=\big[\frac{nT}{N^2},\frac{(n+1)T}{N^2}\big]$, $0\leq n\leq N^2-1$.
Taking $t=\frac{nT}{N^2}$ in \eqref{IntegralFormulationOfY_BSDE} yields \[ Y_B^N\left(\frac{nT}{N^2}\right)=Z_C^N\left(\frac{nT}{N^2}\wedge\tau\right)+\int_0^{\frac{nT}{N^2}}A_NY_B^N(s)ds. \]
Thus, we can write 
\begin{align*}
Y_B^N(t) & = Z_B^N(t\wedge\tau)+\int_0^{\frac{nT}{N^2}}A_NY_B^N(s)ds+\int_{\frac{nT}{N^2}}^tA_NY_B^N(s)ds\\
& = Z_B^N(t\wedge\tau)+\left[Y_B^N\left(\frac{nT}{N^2}\right)-Z_B^N\left(\frac{nT}{N^2}\wedge\tau\right)\right]+\int_{\frac{nT}{N^2}}^tA_NY_B^N(s)ds\\
& = Y_B^N\left(\frac{nT}{N^2}\right)+\int_{\frac{nT}{N^2}}^tA_NY_B^N(s)ds+\tilde{m}_B(t),
\end{align*}
where $\tilde{m}_B(t)=Z_B^N(t\wedge\tau)-Z_B^N\big(\frac{nT}{N^2}\wedge\tau\big)$ defines a mean $0$ martingale, for $t\in I_n(T)$, such that $|\delta\tilde{m}_B(t)|\leq K^{-1}$.
Thus, 
\[ \Vert Y_B^N(t)\Vert_\infty\leq \left\Vert Y_B^N\left(\frac{nT}{N^2}\right)\right\Vert_\infty+cN^2\int_{\frac{nT}{N^2}}^t\Vert Y_B^N(s)\Vert_\infty ds+\Vert \tilde{m}_B(t)\Vert_\infty, \] 
where $c$ is a constant independent of $N$, such that $\Vert A_N\Vert_\infty\leq cN^2$. 
From Gronwall lemma,
\begin{equation}\label{IntermediateBoundUniformOnSmallTimeIntervalForY_B}
\sup_{I_n(T)}\Vert Y_B^N(t)\Vert_\infty\leq \left(\left\Vert Y_B^N\left(\frac{nT}{N^2}\right)\right\Vert_\infty+\sup_{I_n(T)}\Vert \tilde{m}_B(t)\Vert_\infty\right)\text{e}^{cT}.
\end{equation}

As previously, for $i=1,\cdots,N$ and $\theta\in[0,1]$, we set 
\[ m_B(t)=\theta K\tilde{m}_B\left(t,\frac{i}{N}\right), \hspace{0.5cm}\text{for}\hspace{0.3cm}t\in I_n(T)=\left[\frac{nT}{N^2},\frac{(n+1)T}{N^2}\right]. \]  
The new defined mean zero martigale $m_B$ satisfies $|\delta m_B(t)|\leq1$, and by \cref{SecondTypeOfMartingales}, \[ \sum_{\frac{nT}{N^2}\wedge\tau\leq s\leq t\wedge\tau}[\delta m_B(s)]^2-\int_{\frac{nT}{N^2}\wedge\tau}^{t\wedge\tau}g_B^N(s)ds \] is a mean $0$ martingale for $t\in I_n(T)$, where 
\begin{align*}
g_B^N(s) & = \theta^2K\left[2DN^2\big(\cP_{in}u_{B,i+1}^N(s)+u_{B,i}^N(s)\cP_{out}u_{B,i-1}^N(s)\big)+|F_B|_i^2\big(u^N(s)\big)\right]\\
& \leq c\theta^2KN^2=:h_B^N(s),
\end{align*}
and $\int_{\frac{nT}{N^2}\wedge\tau}^{t\wedge\tau}h_B^N(s)ds\leq \int_{I_n(T)}h_B^N(s)ds \leq c\theta^2KT$ for all $t\in I_n(T)$, since $h_B^N$ is positive. The constant $c$ depends on $c_T$ and $M_F(c_T)$. 
\Cref{BlountExpectationEstimate} then implies $\E\big[\text{exp}\big\{m\big(\frac{(n+1)T}{N^2}\big)\big\}\big] \leq \exp\big(c\theta^2KT\big)$. 
One chooses $\theta$ such that 
\begin{align*}
\P\left\{\sup_{I_n(T)}\tilde{m}_B\left(t,\frac{i}{N}\right)>\epsilon_0\right\} & \leq \text{e}^{-\theta K\epsilon_0}\E\left[\text{exp}\left\{m\left(\frac{(n+1)T}{N^2}\right)\right\}\right]\\
& \leq \text{exp}\big[\theta K(c\theta-\epsilon_0)\big] \hspace{1cm} \leq \hspace{1cm} \text{e}^{-\eta\epsilon_0^2K},
\end{align*} 
applying Doob's inequalities. Here, $\eta=\eta\big(T, c_T,M_F(c_T)\big)>0$ and is going to be generic in the following. 
A similar inequality is easily derived for $-\tilde{m}_B\big(t,\frac{i}{N}\big)$. As a result,
\begin{equation}\label{BoundOfTheIntermediateMartingalePartOnSmallTimeIntervalForY_B}
\P\left\{\sup_{I_n(T)}\Vert \tilde{m}_B(t)\Vert_\infty>\epsilon_0\right\}\leq 2N\text{e}^{-\epsilon_0^2K}.
\end{equation}
From \eqref{BoundInTheSupremumNormOnY_BOf_t_InProbability}, \eqref{IntermediateBoundUniformOnSmallTimeIntervalForY_B} and \eqref{BoundOfTheIntermediateMartingalePartOnSmallTimeIntervalForY_B},
\begin{align*}
\P\left\{\text{e}^{-cT}\sup_{I_n(T)}\Vert Y_B^N(t)\Vert_\infty>\epsilon_0\right\} & \leq \P\left\{\left\Vert Y_B^N\left(\frac{nT}{N^2}\right)\right\Vert_\infty+\sup_{I_n(T)}\Vert \tilde{m}_B(t)\Vert_\infty>\epsilon_0\right\}\\
& \leq \P\left\{\left\Vert Y_B^N\left(\frac{nT}{N^2}\right)\right\Vert_\infty>\frac{\epsilon_0}{2}\right\}+\P\left\{\sup_{I_n(T)}\Vert \tilde{m}_B(t)\Vert_\infty>\frac{\epsilon_0}{2}\right\}\\
& \leq 4N\text{e}^{\eta\epsilon_0^2K}.
\end{align*}
						Hence, 
\[ \P\left\{\text{e}^{-cT}\sup_{[0,T]}\Vert Y_B^N(t)\Vert_\infty>\epsilon_0\right\} \leq \sum_{n=0}^{N^2-1}\P\left\{\text{e}^{-cT}\sup_{I_n(T)}\Vert Y_B^N(t)\Vert_\infty>\epsilon_0\right\} \leq 4N^3\text{e}^{-\alpha\epsilon_0^2K} \] 
and it follows that 
\[ \P\left\{\sup_{[0,T]}\Vert Y_B^N(t)\Vert_\infty>\epsilon_0\right\} \leq 4N^3\text{e}^{-\eta K}. \] 
The r.h.s. vanishes, as we are assuming that $K^{-1}\log N\rightarrow 0$. This proves the convergence of the part related to bacteria, and the proof of \cref{TheLLNForSpatialCholera} ends. $\blacksquare$

\section{Further discussion}

In the LLN given by \cref{TheLLNForSpatialCholera}, we consider that
\[ H=\frac{\bH(0)}{N}\rightarrow\infty,\hspace{0.5cm}\frac{H}{K}\rightarrow \text{non-negative constant},\quad \text{and} \quad K^{-1}\log N\rightarrow0 \] 
as $\bH(0),N,K\rightarrow\infty$. We have seen that, due to $\frac{H}{K}>0$, the condition $K^{-1}\log N\rightarrow0$ is equivalent  to $H^{-1}\log N\rightarrow0$. There are many other interesting possibilities of scaling. 

The next one we consider is when
\[ H=\frac{\bH(0)}{N}\rightarrow\infty,\hspace{0.5cm}\frac{H}{K}\rightarrow 0,\hspace{0.5cm}\text{and} \quad K^{-1}\log N\rightarrow0. \] 

\noindent \textbf{\textit{Another LLN.}} This case is very close to the previous one and a similar LLN holds, with two main differences. 
First of all, the equivalence between $K^{-1}\log N\rightarrow0$ and $H^{-1}\log N\rightarrow0$ is lost, since $H/K\rightarrow0$. Therefore, one needs consider the assumption that both conditions hold, instead of any of them as previously. Furthermore, the limit process that was described by \eqref{RenormalizedDerterministicPDE} changes to:
 
\begin{equation}\label{RenormalizedDerterministicPDEwhenHumanIGrowSlowerThanBacteria}
 \left\{  \begin{array}{lcl}
			   \displaystyle \frac{\p S(t,x)}{\p t} &=&\displaystyle  \mu I(t,x) + (\mu+\rho) R(t,x) - \beta\frac{B(t,x)}{1 + B(t,x)}S(t,x) \vspace{0.2cm} \\
			   \displaystyle \frac{\p I(t,x)}{\p t} &=& \displaystyle \beta\frac{B(t,x)}{1 + B(t,x)}S(t,x)  - (\gamma + \alpha + \mu)I(t,x) \vspace{0.2cm}\\
			   \displaystyle \frac{\p R(t,x)}{\p t} &=& \displaystyle \gamma I(t,x) - (\rho + \mu)R(t,x) \vspace{0.2cm}\\
			   \displaystyle \frac{\p B(t,x)}{\p t} &=& \displaystyle \fD  \frac{\p^2 B(t,x)}{\p x^2} - \nu \frac{\p B(t,x)}{\p x} -\mu_B B(t,x), 
		\end{array}
\right.
\end{equation}

In fact, $H/K\rightarrow0$ means that the human population is negligible besides that of bacteria, and the contribution of the former to the latter ---through infectious--- vanishes at the limit. The interaction between the two populations is one-way (to be compared. The population of vibrios evolves independently from that of human, and influences the evolution of that latter. The result in the present context is the following.

\begin{theorem}\label{LLNwithHumanGrowingSlowerThanBacteria}
Consider a sequence $u^N=\big(u_S^N,u_I^N,u_R^N,u_B^N\big)$ of Markov processes starting at $u^N(0)=\big(u_S^N(0),u_I^N(0),u_R^N(0),u_B^N(0)\big)$, with the infinitesimal generators $\cA^N$ given by \eqref{StochasticSpatialGenerator}. Assume that: \vspace{0.2cm}\par 
(i) $N,H,K\longrightarrow\infty$ in such a way that $H/K\rightarrow 0$, $K^{-1}\log N\rightarrow0$ and $H^{-1}\log N\rightarrow0$.\par 
(ii) Assumptions of \cref{Well-posednessOfCompactRenormalizedDeterministicPDE} hold, so that  \eqref{RenormalizedDerterministicPDEwhenHumanIGrowSlowerThanBacteria} is well-posed, with solution $v=(S,I,R,B)\in C\left(\R_+;\big[C(J)\big]^3\times C^3(J)\right)$. \par  
(iii) $u^N(0)\in \big(H^{-1}\N\big)^3\times K^{-1}\N$ and $\Vert u^N(0)-v(0)\Vert_{\bold C(J)}\longrightarrow0$ in probability.\vspace{0.2cm}\par 
\noindent Then, for all $T>0$, \[ \sup_{t\leq T}\Vert u^N(t)-v(t)\Vert_{\bold C(J)}\longrightarrow0\hspace{0.3cm}\text{in probability}. \]
\end{theorem}

The proof is a straightforward adaptation of that of \cref{TheLLNForSpatialCholera}. $\blacksquare$ \\

\noindent \textbf{\textit{A divergent configuration.}} The last case in the context of infinite local human population ($H\rightarrow\infty$) is when $H/K\rightarrow \infty$. In this case, the human population is too abundant. Their contribution to the population of vibrios is of order of infinity, and the population of bacteria then explodes. As a result, the rate $\lambda_{(B)}=\beta\frac{B}{1+B}$ at which susceptible people become infected is of order of $\beta$ at the limit, and the evolution of human population is independent of that of bacteria. Here also, the interaction are one-way, with human influencing bacteria. The global dynamic depends upon the dominant population essentially. \\

\noindent \textbf{\textit{A hybrid approximation.}} Now, we discuss a second class of scaling, when the local human population is of order of a constant, with $H>0$ naturally. This immediately yields $H/K\rightarrow0$, and we already know that the contribution of human to the population of vibrios will vanish at the limit, as in \eqref{RenormalizedDerterministicPDEwhenHumanIGrowSlowerThanBacteria}. 
However, it is difficult to obtain a limit for human compartments, even formally, in the generator. 

We notice that in all what precedes, no effective spatial behaviour has been considered for human directly. The dependence of human with space holds through that of bacteria. It might be interesting to introduce it. The formalism of \cite{Nankep2018} Chapter 3 appears to be adequate. In fact, the motions of bacteria happen at a spatial microscopic level whereas they happen at a spatial macroscopiec level for human who have much more bigger sizes.  

Accordingly, a natural way to proceed is to consider human spatial dynamics on a macroscopic discretization of the domain, which is fixed, and totally independent of the parameter $N$ of the microscopic discretization that we introduced earlier. Microsites ---or sites or nodes as introduced previously--- are now distinguishable from macrosites. These latter may refer to regions in the domain ---such as communities in the real life---, between which human are transported. On the one hand, each macrosite independently undergoes intraregion interactions (birth, death, infection, recovering) with rates depending on the entire region or macrosite, with jumps given by a function that depends on the region. On the other hand, the regions communicate in a way that has to be specified. One should also keep in mind that the dynamic of human remains coupled to that of bacteria.

From Chapter 3 of \cite{Nankep2018}, we know that a Piecewise determistic Markov Process (PDMP) shall be obtained at the limit in this context, as soon as human jumps are convergent and the rates describing human compartments dynamics are smooth enough\footnote{For instance, jump rates as considered in the present article are sufficiently smooth.}. In that limit, human population follow a pure jump dynamic whose parameters depend upon bacteria population. Human compartments then represent the discrete component of the limiting PDMP, while vibrio population represents its continuous component. Between human compartments consecutive jumps, the vibrio population follows an appropriate reaction-advection-diffusion PDE, parametered by the state taken by the human compartments between the considerated jumps.

\appendix
\section{Appendix}

\subsection{On semigroups and operators in Banach spaces}

We start with some insight into unbounded operators in Banach spaces. Our aim is to introduce and characterize the semigroup generated by some particular classes of operators. For more details and precisions, we refer to \cite{Cazenave1998} or \cite{Henry1981} among others. Let $X$ be a Banach space endowed with the norm $\Vert\cdot\Vert_X$. We consider real Banach spaces by default, and consider their complexification when the context requires a complex field (e.g. for spectral theory).

\begin{definition}(Linear unbounded operator) \normalfont\par 
\noindent A linear unbounded operator in $X$ is a pair $(D,L)$, where $D$ is a linear subspace of $X$ and $L$ is a linear mapping $D\rightarrow X$. 

The (unbounded) operator $L$ can be either "bounded", if there exists $c>0$ such that \[ \Vert Lu\Vert_X\leq c,\hspace{0.3cm}\forall u\in\{x\in D,\Vert x\Vert_X\leq 1\}, \] or "not bounded" otherwise.

The graph $G(L)$ and the range $\cR(L)$ of $L$ are the linear subspaces of $X\times X$ and $X$ respectively, defined by 
\[ \begin{array}{l}
\disp G(L):=\left\{(u,f)\in X\times X: u\in D\text{ and } f=Lu\right\},\vspace{0.1cm}\\
\disp \cR(L):=L(D).
\end{array} \]
If $G(L)$ is a closed subspace of $X\times X$, then $L$ is said to be closed. 

We often denote by $L$ the operator, and by $D=D(L)$ its domain. However, when one defines an operator it is necessary to define its domain. If this latter is dense in $X$ ($\overline{D(L)}=X$), the operator is said to be densely defined.

As it will turn out, having a closed graph and a dense domain confers nice properties to (unbounded) operators.

\end{definition}

\begin{definition}\textbf{($m-$dissipativity)}\normalfont\par 

\noindent Consider an (unbounded) operator $L$ in $X$, and the following conditions:  \vspace{0.1cm}\par 
\noindent\hspace{0.2cm} \textbf{(mD1)} $\disp \Vert u-\lambda Lu\Vert_X\geq\Vert u\Vert_X$, for all $u\in D(L)$ and all $\lambda>0$ (dissipativity).\par 
\noindent\hspace{0.2cm} \textbf{(mD2)} For all $\lambda>0$ and all $f\in X$, there exists $u\in D(L)$ such that $u-\lambda Lu=f$.\vspace{0.1cm}\par 

$\bullet$ The operator $L$ is dissipative if it satisfies the first condition (mD1).\par 
$\bullet$ The operator $L$ is $m-$dissipative if it satisfies both conditions (mD1) and (mD2).
\end{definition}

It can be showed (see e.g. Proposition 2.2.6, p. 19 of \cite{Cazenave1998}) that if the condition (mD1) holds, it is sufficient to find some $\lambda_0>0$ such that for all $f\in X$, a solution of $u-\lambda_0 Lu=f$ exists, in order for the operator $L$ to be $m-$dissipative.\par 

\medskip
Next, we introduce some tools and results related to $m-$dissipative operators. Detailed proves can be found, e.g., in the Chapter 2 of \cite{Cazenave1998}.

\begin{proposition}\label{some_properties_of_m-dissipative_operators}\normalfont
\noindent Let $L$ be an $m-$dissipative operator in $X$:\vspace{0.2cm}\par 

\textbf{(i)} For all $f\in X$ and all $\lambda>0$, there exists a unique solution  to the equation $u-\lambda Lu=f$, that we denote by \[ J_\lambda f\equiv(I_d-\lambda L)^{-1}, \] where $I_d$ is the identity operator on $X$. For $\lambda>0$, we also introduce the operator \[ L_\lambda=LJ_\lambda=\frac{J_\lambda-I_d}{\lambda}. \] Therefore, $J_\lambda, L_\lambda\in \cL(X)$, and in addition, $\Vert J_\lambda\Vert\leq 1$ and $\Vert L_\lambda\Vert\leq 2/\lambda$. \vspace{0.2cm}\par 

\textbf{(ii)} The operator $L$ is closed. For every $u\in D(L)$, the graph norm of $u$ is given by \[ \Vert u\Vert_{D(L)}:=\Vert u\Vert_X+\Vert Lu\Vert_X. \] Then $\disp\big(D(L),\Vert\cdot\Vert_{D(L)}\big)$ is a Banach space, and $L\in \cL(D(L),X)$. Moreover, \[ \lim_{\lambda\downarrow0}\Vert J_\lambda u-u\Vert_X=0\hspace{0.3cm} \text{for all}\hspace{0.2cm} u\in\overline{D(L)}. \] Furthermore, if $L$ is densely defined, then \[ \lim_{\lambda\downarrow0}\Vert L_\lambda u-Lu\Vert_X=0\hspace{0.3cm} \text{for all}\hspace{0.2cm} u\in D(L). \] 
\end{proposition}

\medskip
Let us switch to the notion of semigroup. 
\begin{definition}\label{semigroup_definition}\textbf{\texttt{(semigroup: contraction and strong continuity)}}\normalfont\\
A one-parameter family $\{S(t)\}\equiv\{S(t),t\geq0\}$ of linear operators on $X$ is called a semigroup on $X$ if 
\[
\disp \left\{
\begin{array}{l}
S(t)\in\cL(X),\vspace{0.1cm}\\
S(0)=I_d,\vspace{0.1cm}\\
S(t+s)=S(t)S(s),\hspace{0.2cm} \forall s,t\geq0.
\end{array} 
\right.
\]
A semigroup $\{S(t)\}$ on $X$ is said to be a contraction semigroup if \[ \Vert S(t)\Vert\leq 1\hspace{0.2cm} \forall t\geq0. \] It is said to be strongly continuous if \[ \lim_{t\rightarrow0}\Vert S(t)u-u\Vert_X\rightarrow0\hspace{0.2cm}\forall u\in X, \] which means that for all $u\in X$, $t\mapsto S(t)u$ belongs to $C(\R_+,X)$.\par 
\end{definition}
We have followed \cite{Kurtz1986} in the definition above. We remark that terminologies may vary with authors. However, the essential ideas remain the same. For example, in \cite{Cazenave1998} (see Definition 3.4.1 p.39), the authors consider "strong continuity" as an intrinsic property of a semigroup and do not emphasize on that point when denoting the semigroup.\par 

Associated with semigroups are different operators. We present some of them throughout the present paper. Let us get started with those generated by $m-$dissipative operators with dense domain. Let $L$ be such an operator. For $\lambda>0$, we consider the operators $J_\lambda$ and $L_\lambda$ introduced in $\cref{some_properties_of_m-dissipative_operators}$ (i). Then we set \[ S_\lambda(t)=e^{tL_\lambda}, \] and fix $T>0$. The following holds.

\begin{proposition}\label{semigroup_generated_by_an_m-dissipative_operator}\textbf{\texttt{(Theorem 3.1.1, p.33, Chapter 3 of \cite{Cazenave1998})}}\\
For all $w\in X$, the sequence $u_\lambda(t)=S_\lambda(t)w$ converges uniformly on bounded intervals of $[0,T]$, to a function $u\in C(\R_+,X)$, as $\lambda\downarrow0$. We set \[ S(t):=e^{tL}\hspace{0.4cm}\text{and}\hspace{0.4cm}S(t)w=u(t),\hspace{0.5cm}\forall w\in X, t\geq0. \] Then $\{S(t)\}$ defines a (one-parameter) semigroup of contraction on $X$.\par 
In addition, for all $w\in D(L)$, $u(t)=S(t)w$ is the unique solution of the problem 
\[ \left\{
\begin{array}{l}
u\in C(\R_+,D(A))\cap C^1(\R_+,X),\vspace{0.1cm}\\
u'(t)=Lu(t),\hspace{0.2cm}\forall t\geq0,\vspace{0.1cm}\\
u(0)=w.
\end{array} 
\right.
\]
Finally, the semigroup $\{S(t)\}$ commutes with $L$ in the sense \[ S(t)Lw=LS(t)w \] for all $w\in D(L)$ and $t\geq0$.
\end{proposition}

Next, there is:
\begin{definition}\label{infinitesimal_generator_of_a_contraction_semigroup}\textbf{\texttt{(The infinitesimal generator)}}\normalfont\\
The (infinitesimal) generator of a semigroup $\{S(t)\}$ is the linear operator $\cL$ on $X$ defined by \[ D(\cL):=\left\{u\in X:\frac{S(h)u-u}{h} \text{ has a limit in } X \text{ as } h\downarrow0\right\}, \] and \[ \cL u:=\lim_{h\downarrow0}\frac{S(h)u-u}{h}\hspace{0.2cm}\forall u\in D(\cL). \]
It is well known that if $\{S(t)\}$ is a strongly continuous semigroup of contraction, then its (infinitesimal) generator is $m-$dissipative and densely defined (see e.g. Proposition 3.4.3., p. 39, of \cite{Cazenave1998}).
\end{definition}

With $\cref{semigroup_generated_by_an_m-dissipative_operator}$ in mind, we remark that an $m-$dissipative operator needs not be the (infinitesimal) generator of its associated semigroup. However, for a particular class of semigroups, that identification is certain. This is made precise by the so-called Hille-Yosida-Phillips theorem.

\begin{proposition}\label{the_Hille_Yosida_Phillips_theorem}\textbf{\texttt{(The Hille-Yosida-Phillips Theorem)}}\normalfont\\
A linear operator $L$ is the (infinitesimal) generator of a strongly continuous semigroup of contraction in $X$ if and only if $L$ is $m-$dissipative and densely defined.
\end{proposition}

\subsection{On the properties of the operator $A=\fD\Delta + \nu\nabla$}\label{PropertiesOfTheOperator_A}

In this section, we aim to prove that the operator $A$ defined on $C(J)$ is $m$-dissipative and densely defined 
We entirely rely on the approach used in \cite{Cazenave1998} which is in three steps. At first, we prove that the operator is $m$-dissipative with dense domaine in $L^2(J)$. Then, we deduce $m$-dissipativity in $L^\infty$. In that case, the domain is not dense. Finally, we conclude in $C$ framework. Recall that we are considering $1$-periodic functions, and there is no deal with boundaries, since there is no boundary. \\

\noindent \underline{\textbf{\textit{Step 1 : $L^2$-theory.}}} Consider the operator $A_2$ on $L^2(J)$, by:
\[
\left\{
\begin{array}{l}
\cD(A_2) = \big\{ u\in H^1(J): \Delta u\in L^2(J) \big\}\vspace{0.2cm}\\
A_2u = \fD\Delta u - \nu\nabla u, \quad \forall u\in \cD(A_2),
\end{array}
\right.
\]
where $\fD$, $\nu>0$. 
We prove the following:

\begin{lemma}\label{PropertiesOf_A_onL2}
The operator $A_2$ is  m-dissipative and densely defined.
\end{lemma}

\begin{proof}

Since $C^\infty(J)$ is dense in $L^2(J)$ and $C^\infty(J)\subset \cD(A_2)$, it follows that $\cD(A_2)$ is dense in $L^2(J)$.
It remains to prove that $m$-dissipativity.\par  

\medskip
It is not difficult to see that $\langle A_2u,u\rangle\leq 0$ for all $u\in \cD(A_2)$. In fact, observing that $\langle v,\Delta u\rangle=-\langle\nabla v,\nabla u\rangle$ for all $u\in \cD(A)$ and $v\in H^1(J)$ (see Lemma 2.6.2 of \cite{Cazenave1998}), and that $\nabla$ is skew-adjoint on $L^2(J)$,
one easily obtains that $A_2$ is negative definite by taking $v=u$. Therefore $A_2$ is dissipative, by Proposition 2.4.2 of \cite{Cazenave1998}. We will conclude using Lax-Milgram theorem. 

\medskip
Consider the coercive continuous bilinear form $b$ on $H^1(J)$ 
defined by 
\[ b(u,v)=\langle u,v\rangle+\fD\langle \nabla u,\nabla v\rangle+\nu\langle u,\nabla v\rangle. \]
Bilinearity is evident. Let $u,v\in H^1(J)$. By Schwartz inequality,
\begin{align*}
|b(u,v)| & \leq \Vert u\Vert_2\Vert v\Vert_2+\fD\Vert\nabla u\Vert_2\Vert\nabla v\Vert_2+\nu\Vert u\Vert_2\Vert\nabla v\Vert_2\vspace{0.2cm}\\
& \leq 3\max(1,\fD,\nu)\Vert u\Vert_{H^1(J)}\Vert v\Vert_{H^1(J)},
\end{align*}
and continuity follows. Furthermore, 
\[ b(u,u) \quad = \quad \Vert u\Vert_2^2+\fD\Vert\nabla u\Vert_2^2 \quad \geq \quad \min(1,\fD)\Vert u\Vert_{H^1(J)} \]
yields coerciveness. Now, let $f\in L^2(J)\subset H^{-1}(J)$. There exists a unique $u\in H^1(J)$ such that $b(u,v)=\langle f,v\rangle_{H^1(J)}$ for all $v\in H^1(J)$, thanks to Lax-Milgram theorem. 
From Proposition 8.14, Chapter 8 of \cite{Brezis2011}, since $J=[0,1]$ is bounded, there exists a ---non necessarily unique--- $f_1\in L^2(J)$ of $f$, such that $\langle f,v\rangle_{H^1(J)}=\langle f_1,v\rangle$. This allows one to identify $f$, viewed as an element of the dual space $H^{-1}(J)$ of $H^1(J)$, with the distribution $-f_1^{'}$\footnote{The distribution $-f_1^{'}$ is the linear functional on $C^\infty(J)$ defined by $v\mapsto \langle f_1,\nabla v\rangle$.} (see Remark 20, Chapter 8 \cite{Brezis2011}). In the following, we use that identification and denote it by $-f$. \par 

\medskip
We have proved that for all $f\in L^2(J)$, there exists a unique $u\in H^1(J)$ such that \[ \langle u,v\rangle-\fD\langle u,\Delta v\rangle+\nu\langle u,\nabla v\rangle=\langle f,v\rangle,\quad \text{for all}\quad v\in H^1(J). \]
Thus \[ u-(\fD \Delta u-\nu\nabla u)=f \] in the sense of distributions. Since $u\in H^1(J)$ in addition, we obtain $u\in \cD(A_2)$ and $u-A_2u=f$. Therefore, $A_2$ is $m$-dissipative.

\end{proof}

\noindent \underline{\textbf{\textit{Step 2 : $L^\infty$-theory.}}} Consider the operator $A_2$ on $L^2(J)$, by:
\[
\left\{
\begin{array}{l}
\cD(A_\infty) = \big\{ u\in H^1(J)\cap L^\infty(J): \nabla u\in L^\infty(J) \text{ and } \Delta u\in L^\infty(J) \big\}\vspace{0.2cm}\\
A_\infty u = \fD\Delta u - \nu\nabla u, \quad \forall u\in \cD(A_\infty),
\end{array}
\right.
\]
where $\fD$, $\nu>0$. Then comes the next result.

\begin{lemma}\label{PropertiesOf_A_onLinfty}
The operator $A_\infty$ is m-dissipative in $L^\infty(J)$.
\end{lemma}

\subsection{Proof of \Cref{DiscreteApproximationOfTheLimit}}\label{ProofOfDiscreteApproximationOfTheLimit}

Since $\tilde{A}_N=\text{diag}(0,0,0,-\nu\nabla_N+\fD\Delta_N)$ is linear, it is Lipschitz. Next, the vector field $F=(F_S,F_I,F_R,F_B)$ is locally Lipschitz continuous. Thus, the initial value problem \eqref{DiscreteVersionOfPDE} has a unique local solution $v^N$, thanks to the Picard-Lindelöf theorem. That solution satisfies \eqref{MildRepresentationOfTheDiscreteVersion}. The bound \eqref{BoundOfTheDiscreteVersion} is obtained from \eqref{ReactionDebitBoundProperties} (iii) and Gronwall lemma as in the previous sections, and we deduce that $v^N$ is in fact a global solution.

Now, let $T>0$ be fixed. From \eqref{BoundednessOfTheTruncatedProcess}, we may assume that $F$ is globally Lipschitz and we choose $L$ such that $\Vert F(u)-F(\tilde{u})\Vert_{\bold C(J)}\leq L\Vert u-\tilde{u}\Vert_{\bold C(J)}$, provided $\Vert u\Vert_{\bold C(J)}\leq c_T$. 
In the rest of the proof, $c$ denotes a generic constant that may depend upon $T$, $c_2$ and $L$.
From \eqref{MildRepresentationOfDeterministicModel} and \eqref{MildRepresentationOfTheDiscreteVersion}, we have
\[ v^N(t)-v(t)=\tilde{T}_N(t)\tilde{P}_Nv(0)-\tilde{T}(t)v(0)+\int_0^t\left(\tilde{T}_N(t-s)F(v^N(s))-\tilde{T}(t-s)F(v(s))\right) ds \]
for all $t\geq0$. Then, \\ 

$\displaystyle\hspace{0.5cm} \left\Vert v^N(t)-v(t)\right\Vert_{\infty,\infty}\leq \left\Vert\tilde{T}_N(t)\tilde{P}_Nv(0)-\tilde{T}(t)v(0)\right\Vert_{\bold C(J)}$\vspace{0.1cm}\par 

$\displaystyle \hspace{5cm}+\int_0^t\left\Vert\tilde{T}_N(t-s)\big(F(v^N(s)) - \tilde{P}_N F(v(s))\big)\right\Vert_{\bold C(J)}ds$\vspace{0.1cm}\par 

$\displaystyle \hspace{5cm}+\int_0^t\left\Vert\tilde{T}_N(t-s)\tilde{P}_N F(v(s))-\tilde{T}(t-s)F(v(s))\right\Vert_{\bold C(J)}ds$.\\

\noindent Observing that $\bold H^N$ is stable by $F$ and the projection $\tilde{P}_N$ is a contracting linear operator on $\bold C_p(J)$, the second and third terms on the r.h.s. of the above inequality satisfy
\[ T_2(N) \leq \int_0^t\left\Vert \tilde{P}_N\big(F(v^N(s))-F(v(s))\big)\right\Vert_{\bold C(J)} ds \leq cL\int_0^t\left\Vert v^N(s)-v(s)\right\Vert_{\bold C(J)}ds, \] and\par 
$\displaystyle\hspace{1.5cm} T_3(N) = \int_0^t\left\Vert T_N(t-s)P_NF(v(s))-T(t-s)F(v(s))\right\Vert_\infty ds$.\\ 

\noindent Thus, 
taking the supremum in $t$ on $[0,T]$ and using Gronwall lemma leads to\\

\noindent $\displaystyle\hspace{1cm} \sup_{t\leq T}\left\Vert v^N(t)-v(t)\right\Vert_{\infty,\infty}$\par 
$\displaystyle\hspace{1.3cm}\leq \left(\sup_{t\leq T}\left\Vert\tilde{T}_N(t)\tilde{P}_Nv(0)-\tilde{T}(t)v(0)\right\Vert_{\bold C(J)}\right.$\par 

$\displaystyle \hspace{2.3cm} \left.+\int_0^{T}\sup_{t\leq T}\left(\left\Vert T_N(t-s)P_NF(v(s))-T(t-s)F(v(s))\right\Vert_\infty\mathds{1}_{(s\leq t)}\right)ds\right)\times \text{e}^{cLT}$.\\

\noindent Firstly, \[ \sup_{t\leq T}\left\Vert\tilde{T}_N(t)\tilde{P}_Nv(0)-\tilde{T}(t)v(0)\right\Vert_{\bold C(J)}\longrightarrow0, \]
since $v(0)=v_0\in \big[C(J)\big]^3\times C^3(J)$ yields $\Vert \tilde{A}_N\tilde{P}_Nv_0-\tilde{A}v_0\Vert_E\rightarrow0$ (see \cite{Kato1966}, chapter 9, section 3). 
Secondly, we fix $t\in [0,T]$ and $s\in [0,t]$. Since $F\big(v(s)\big)\in \big[C(J)\big]^3\times C^3(J)$, the same argument as previously yields \[ \sup_{t\leq T}\left\Vert T_N(t-s)P_NF(v(s))-T(t-s)F(v(s))\right\Vert_{\bold C(J)}\longrightarrow0, \] and we conclude by dominated convergence. $\square$

\vspace{1cm}
\bibliographystyle{alpha}
\bibliography{cholera_modeling_LLN_LDP}

\end{document}